\documentclass[a4paper]{amsart}

\usepackage{amsfonts}
\usepackage{amssymb}
\usepackage{amsthm}
\usepackage{amsmath}
\usepackage[all]{xy}

\SelectTips{cm}{}

\theoremstyle{plain}
\newtheorem{theorem}{Theorem}[section]
\newtheorem{proposition}[theorem]{Proposition}
\newtheorem{corollary}[theorem]{Corollary}
\newtheorem{conjecture}[theorem]{Conjecture}
\newtheorem{lemma}[theorem]{Lemma}

\theoremstyle{definition}
\newtheorem{definition}[theorem]{Definition}

\theoremstyle{remark}
\newtheorem{remark}[theorem]{Remark}
\newtheorem{example}[theorem]{Example}

\numberwithin{equation}{section}

\def\C{\mathbb C}
\def\N{\mathbb N}
\def\R{\mathbb R}
\def\Z{\mathbb Z}

\newcommand{\coker}{\mathop{\mathrm{coker}}\nolimits}
\newcommand{\colim}{\mathop{\mathrm{colim}}\nolimits}
\newcommand{\cone}{\mathop{\mathrm{cone}}\nolimits}
\newcommand{\map}{\mathop{\mathrm{map}}\nolimits}
\newcommand{\diag}{\mathop{\mathrm{diag}}\nolimits}
\newcommand{\im}{\mathop{\mathrm{im}}\nolimits}
\newcommand{\ind}{\mathop{\mathrm{ind}}\nolimits}
\newcommand{\res}{\mathop{\mathrm{res}}\nolimits}
\newcommand{\tr}{\mathop{\mathrm{tr}}\nolimits}

\def\d{\, \textup{d}}
\def\id{\textup{id}}
\def\lcm{\textup{lcm}}
\def\NZD{\textup{NZD}}
\def\pr{\textup{pr}}
\def\Wh{\textup{Wh}}

\title{$L^2$-Alexander invariant for knots}

\author{J\'er\^ome Dubois}
\email{dubois@math.jussieu.fr}
\address{Institut de Math\'ematiques de Jussieu -- Paris Rive Gauche, Universit\'e Paris Diderot--Paris 7, UFR de Math\'ematiques, B\^atiment Sophie Germain Case 7012, 75205 Paris Cedex 13, France}

\author{Christian Wegner}
\email{wegner@him.uni-bonn.de}
\address{Hausdorff Research Institute for Mathematics (HIM), Universit\"at Bonn, Poppelsdorfer Allee 45, 53115 Bonn, Germany}


\begin{document}

\maketitle

\begin{abstract}
This paper deals with the study of a new family of knot invariants: the $L^2$-Alexander invariant. A main result is to give a method of computation of the $L^2$-Alexander invariant of a knot complement using any presentation of default 1 of the knot group.
\end{abstract}

\section{Introduction}

The Alexander polynomial is the first polynomial knot invariant in the history of knot theory. It was introduced by J.~W.~Alexander~\cite{Alexander} in 1928 and appears now in many different flavors: using the Fox calculus, using Seifert matrices, using skein relations, etc. In the sixties, J.~Milnor~\cite{Milnor:1962} gave a spectacular interpretation of the (usual) Alexander polynomial as a kind of Reidemeister torsion, an abelian one, and proved again some deep properties such as symmetric property of the Alexander polynomial.

In a completely different area, $L^2$-invariants was introduced by M.~Atiyah in~\cite{A} where he used von Neumann's concept of ``continuous dimension".

The Milnor--Reidemeister torsion is defined using matrices, i.e. operators with finite spectrum, and using the usual notion of determinant. The analytic Ray--Singer torsion is defined for operators with (infinite) discrete spectrum, and it is well--known that for closed three--dimensional manifolds analytic and Reidemeister torsion are equal by the celebrated  theorem of Cheeger--M\"uller (see~\cite{Ch, Mul}). For certain operators whose spectrum is no more discrete but continuous, the notion of $L^2$-torsion has been introduced around 15 years ago by Carey--Mathai, Lott, L\"uck--Rothenberg, Novikov--Shubin (see in particular L\"uck's monograph~\cite{Luc02} for a complete history of the story). One of the most significant results in this field is that the $L^2$-torsion of a hyperbolic three--dimensional manifold is proportional to the hyperbolic volume of the manifold (in fact equal up to a factor $-\frac{1}{6\pi}$). This fundamental result is due to W. L\"uck and T. Schick~\cite{LucSch99}.

As a generalization of the Milnor--Reidemeister torsion, the notion of twisted Alexander polynomial has been introduced in the nineties by X.-S. Lin, and next generalized and studied by many authors: one could refer to the excellent survey by Friedl and Vidussi~\cite{FV} for a complete bibliography. In 2006, Li and Zhang~\cite{LZ06a} introduced the notion of $L^2$-Alexander invariants for knots, which is a sort of Alexander type invariant (twisted in an abelian way) but using in its definition the Fuglede--Kadison determinant instead of the usual determinant.
Further observe that the $L^2$-Alexander  invariant of a knot evaluated at $t = 1$ is precisely the $L^2$-torsion of the knot complement.

One of the main result of this paper is to give a method to compute the $L^2$-Alexander invariant of a knot using any presentation of default 1 of its group and the associated Fox matrix (see Theorem~\ref{mainthm}).
This result can be considered as an $L^2$-version of a well--known method to compute the usual Alexander polynomial by using Fox differential calculus (see for example \cite[Theorem 9.10]{BurdeZieschang}).
The idea of the proof is to relate the $L^2$-Alexander invariant of a knot to a generalization of the $L^2$-torsion --- the so--called \emph{weighted} $L^2$-torsion --- of the universal covering of the knot complement. A detailed study of the weighted $L^2$-torsion yields Theorem~\ref{mainthm}. As an application we explicitly compute the $L^2$-Alexander invariant for torus knots.
A strategy to gain further information about the $L^2$-Alexander invariant is to use a geometrical toral splitting of the knot complement along disjoint incompressible 2-sided tori into Seifert and hyperbolic pieces and to study the weighted $L^2$-torsion of these pieces (see Proposition~\ref{prop_pieces} and Remark~\ref{rem_pieces}).

The paper is organized as follows. Section~\ref{S:background} deals with the needed backgrounds on $L^2$-invariants, we especially give a precise definition of the Fuglede-Kadison determinant. In Section~\ref{S:Alexander}, we define the $L^2$-Alexander invariant which can be considered as a $L^2$-version of the usual Alexander invariant and state our main result: Theorem~\ref{mainthm}. The proof of our main theorem uses the notion of weighted $L^2$-invariants, this notion is introduced in Section~\ref{sec_weighted_L2-invariants} and applied to knot complements in Section~\ref{sec_weighted_L2-torsion}. In Section~\ref{S:computations}, we explicitly compute the $L^2$-Alexander invariant for torus knots. Finally in the Appendix we have put together the detailed study of the weighted $L^2$-invariants.

\section{Background on $L^2$-invariants}
\label{S:background}

In this section, we give a short introduction to $L^2$-invariants. We define $L^2$-Betti numbers, Novikov--Shubin invariants and $L^2$-torsion of finite CW-complexes. For more information on $L^2$-invariants we refer to L\"uck's book \cite{Luc02}.

We start with some elementary definitions. For a discrete group $G$, the Hilbert space $l^2(G)$ is defined as the completion of the complex group ring ${\C}G$ with respect to the inner product
\[
\langle \sum_{g \in G} c_g \cdot g \, , \, \sum_{g \in G} d_g \cdot g \rangle := \sum_{g \in G} \overline{c_g} \cdot d_g.
\]
The von Neumann algebra ${\mathcal N}(G)$ is the algebra of all bounded linear endomorphisms of $l^2(G)$ that commute with the left $G$-action.
The trace of an element $\phi \in {\mathcal N}(G)$ is defined by $\tr_{{\mathcal N}(G)}(\phi) := \langle \phi(e) \, , \, e \rangle$ where $e \in \C G \subset l^2(G)$ denotes the unit element. We can extend this trace to $n \times n$-matrices over ${\mathcal N}(G)$ by considering the sum of the traces of the entries on the diagonal.

\begin{definition}[Hilbert ${\mathcal N}(G)$-module]
A \emph{finitely generated Hilbert ${\mathcal N}(G)$-module} $V$ is a Hilbert space with a linear left $G$-action such that there exists a $\C G$-linear embedding of the Hilbert space into an orthogonal direct sum of a finite number of copies of $l^2(G)$.
A \emph{morphism of Hilbert ${\mathcal N}(G)$-modules} is a bounded $G$-equivariant operator. A \emph{weak isomorphism} is a morphism which is injective and has dense image.
\end{definition}

By the following definition we can assign to a finitely generated Hilbert ${\mathcal N}(G)$-module a dimension which satisfies faithfulness, monotony, continuity and weak exactness (see \cite[Theorem 1.12]{Luc02}).

\begin{definition}[von Neumann dimension]
Let $V$ be a finitely generated Hilbert ${\mathcal N}(G)$-module.
Choose any orthogonal $G$-equivariant projection $\pr \colon l^2(G)^n \to l^2(G)^n$ whose image is isometrically $G$-isomorphic to $V$.
The \emph{von Neumann dimension} of $V$ is given by
\[
\dim_{{\mathcal N}(G)}(V) := \tr_{{\mathcal N}(G)}(\pr) \in [0,\infty).
\]
\end{definition}

Let $X$ be a finite connected CW-complex with fundamental group $G$. Then we obtain a Hilbert ${\mathcal N}(G)$-chain complex $C^{(2)}_*({\tilde X}) := l^2(G) \otimes_{\Z G} C_*({\tilde X})$, where $C_*({\tilde X})$ is the cellular chain complex of the universal covering of $X$. Notice that $C_n({\tilde X})$ is a finite free $\Z G$-module with basis given by a cellular structure of $X$.

\begin{definition}[$L^2$-homology, $L^2$-Betti number]
We define the \emph{$n$-th (reduced) $L^2$-homology} and the \emph{$n$-th $L^2$-Betti number} of a finite connected CW-complex $X$ by
\begin{eqnarray*}
H_n^{(2)}({\tilde X}) &:=& \ker(c_n^{(2)}({\tilde X})) / \overline{\im(c_{n+1}^{(2)}({\tilde X}))},\\
b_n^{(2)}({\tilde X}) &:=& \dim_{{\mathcal N}(G)}(H_n^{(2)}({\tilde X})).
\end{eqnarray*}
\end{definition}

The basic properties of the $L^2$-Betti numbers like homotopy invariance, Euler-Poincar{´e} formula or multiplicativity under finite coverings are described in \cite[Theorem 1.35]{Luc02}.

For the definition of the Novikov--Shubin invariants and the $L^2$-torsion we need \emph{spectral density functions}.
The spectral density function of a morphism of finitely generated Hilbert ${\mathcal N}(G)$-modules $f\colon U \to V$ is given by
\[
F(f)\colon \R \to [0,\infty), \lambda \mapsto \dim_{{\mathcal N}(G)}\left(\im(E_{\lambda^2}^{f^*f})\right)
\]
where $\{E_\lambda^{f^*f}\colon U \to U \mid \lambda \in \R\}$ denotes the family of spectral projections of the positive endomorphism $f^*f$.
The spectral density function is monotonous and right-continuous. It defines a measure on the Borel $\sigma$-algebra on $\R$ which is uniquely determined by
\[
dF(f)((a,b]) := F(f)(b)-F(f)(a) \mbox{ for } a<b.
\]

\begin{definition}[Novikov--Shubin invariant]
Let $X$ be a finite CW-complex with fundamental group $G$. We define its \emph{Novikov--Shubin invariants} by
\[
\alpha_n({\tilde X}) := \liminf_{\lambda \to 0^+} \frac{\ln(F(c_n^{(2)}({\tilde X}),\lambda)-F(c_n^{(2)}({\tilde X}),0))}{\ln(\lambda)} \in [0,\infty],
\]
if $F(c_n^{(2)}({\tilde X}),\lambda) > F(c_n^{(2)}({\tilde X}),0)$ holds for all $\lambda > 0$. Otherwise we set $\alpha_n({\tilde X}) := \infty^+$ where
$\infty^+$ is a new formal symbol.
\end{definition}

For the basic properties of the Novikov--Shubin invariants like homotopy invariance or invariance under finite coverings we refer to \cite[Theorem 2.55]{Luc02}.

\begin{definition}[Fuglede--Kadison determinant]
Let $f\colon U \to V$ be a morphism of finitely generated Hilbert ${\mathcal N}(G)$-modules. We define the \emph{Fuglede--Kadison determinant} of $f$ by
\[
{\textstyle\det}_{{\mathcal N}(G)}(f) := \exp(\int_{0^+}^\infty \ln(\lambda) \, dF(f)(\lambda))
\]
if $\int_{0^+}^\infty \ln(\lambda) \, dF(f)(\lambda) > -\infty$ and by $\det_{{\mathcal N}(G)}(f) := 0$ otherwise.
\end{definition}

We use this determinant to define the $L^2$-torsion. Some properties of this determinant can be found in \cite[Theorem 3.14]{Luc02}.

\begin{definition}[$L^2$-torsion]
Let $X$ be a finite CW-complex with fundamental group $G$. Suppose that $b_n^{(2)}({\tilde X}) = 0$ and $\det_{{\mathcal N}(G)}(c_n^{(2)}({\tilde X})) > 0$ for all $n$. We define its \emph{$L^2$-torsion} by
\[
\rho^{(2)}({\tilde X}) := - \sum_{n \geq 0} (-1)^n \cdot \ln({\textstyle\det}_{{\mathcal N}(G)}(c_n^{(2)}({\tilde X}))) \in \R.
\]
\end{definition}

The basic properties of the $L^2$-torsion like homotopy invariance or sum formula are described in \cite[Theorem 3.96]{Luc02}.

\section{$L^2$-Alexander invariant for knots}
\label{S:Alexander}

In 2006, Li and Zhang \cite{LZ06a} introduced the notion of $L^2$-Alexander invariant for knots, which is a sort of (twisted) Alexander type invariant but using in its definition the Fuglede-Kadison determinant instead of the usual determinant. We first recall the definition of this $L^2$-Alexander invariant.

Let $K \subset S^3$ be a knot and consider a Wirtinger presentation
\[
P = \big\langle g_1, \ldots, g_k \, \big| \, r_1, \ldots, r_{k-1} \big\rangle
\]
of the knot group $\Gamma = \pi_1(M_K)$, where $M_K = S^3 \setminus V(K)$ denotes the knot exterior.

We have a group homomorphism $\phi \colon \Gamma \to \Z$ given by $g_i \mapsto 1$.
For $t \in \C^*$ we obtain a ring homomorphism
\[
\psi_t \colon \C \Gamma \to \C \Gamma, \sum_{g \in \Gamma} c_g \cdot g \mapsto \sum_{g \in \Gamma} c_g \cdot t^{\phi(g)} \cdot g.
\]
Let $F_j$ ($1 \leq j \leq k$) be the matrix obtained from the Fox matrix $F = \left(\partial r_i / \partial g_l\right)$ by removing its $j$th column.
We obtain a matrix $\psi_t(F_j) \in M((k-1) \times (k-1); \C \Gamma)$ by applying $\psi_t$ entry-wise to the matrix $F_j$.

The \emph{$L^2$-Alexander invariant $\Delta^{(2)}_{K,P}(t)$ of the knot $K$ with respect to the Wirtinger presentation $P$} is defined as the Fuglede--Kadison determinant
\[
\Delta^{(2)}_{K,P}(t) = {\det}_{{\mathcal N}(\Gamma)} \big( r^{(2)}_{\psi_t(F_1)} \colon l^2(\Gamma)^{k-1} \to l^2(\Gamma)^{k-1} \big) \in [0,\infty)
\]
where the map $r^{(2)}_{\psi_t(F_1)}$ is given by right multiplication with the matrix $\psi_t(F_1)$. In the definition we make the following hypothesis:
 \[
(\bullet)\; \text{the map }r^{(2)}_{\psi_t(F_1)} \text{ is injective and } {\det}_{{\mathcal N}(\Gamma)}( r^{(2)}_{\psi_t(F_1)}) > 0.
 \]

 \begin{remark}Hypothesis $(\bullet)$ is technical, here are some remarks about it.
 \begin{enumerate}
  \item Observe that the first hypothesis in $(\bullet)$, $r^{(2)}_{\psi_t(F_1)}$ is injective holds, if and only if some $L^2$-Betti numbers (the weighed $L^2$-Betti numbers defined in Section~\ref{sec_weighted_L2-invariants}) of the universal cover of the knot complement vanish (see in particular Proposition~\ref{prop_Alex-weighted}).
  \item In the case of torus knot, we prove in Section~\ref{S:computations} that hypothesis $(\bullet)$ holds.
  \item For the special values $|t| = 1$, the $L^2$-Alexander invariant is defined and studied in~\cite[Sections 3,5,6]{LZ06a}, whereas~\cite[Section 7]{LZ06a} deals with the general case.
  \item By \cite[Lemma 3.1]{LZ06a} (which also holds for $|t| \neq 1$) one knows that if $r^{(2)}_{\psi_t(F_j)}$ is injective for some $j$ then it is injective for all $j$. In this case $\det_{{\mathcal N}(\Gamma)} (r^{(2)}_{\psi_t(F_j)})$ does not depend on $j$.
\end{enumerate}
\end{remark}

\begin{proposition} \label{prop_|t|}
We have
\[
\Delta^{(2)}_{K,P}(t) = \Delta^{(2)}_{K,P}(|t|)
\]
for all $t \in \C^*$.
\end{proposition}
\begin{proof}
For $c \in U(1)$ the assignment $g_i \mapsto c \cdot g_i$ extends uniquely to a ring homomorphism $\eta_c \colon \C\Gamma \to \C\Gamma$.
Furthermore, this extends to an isometry $\eta_c \colon l^2(\Gamma) \to l^2(\Gamma)$.
Notice that $\psi_t(F_j) = \eta_c(\psi_{|t|}(F_j))$ with $c := t / |t|$.
We conclude
\[
\Delta^{(2)}_{K,P}(t) = {\det}_{{\mathcal N}(\Gamma)} \big( r^{(2)}_{\eta_c \big( \psi_{|t|}(F_j) \big)} \big) = {\det}_{{\mathcal N}(\Gamma)} \big( \diag(\eta_c) \circ r^{(2)}_{\psi_{|t|}(F_j)} \circ \diag(\eta_c^{-1}) \big).
\]
Since $\diag(\eta_c)$ is unitary, $\diag(\eta_c) \circ r^{(2)}_{\psi_{|t|}(F_j)} \circ \diag(\eta_c^{-1})$ and $r^{(2)}_{\psi_{|t|}(F_j)}$ have the same spectral density function. Hence
\[
{\det}_{{\mathcal N}(\Gamma)} \big( \diag(\eta_c) \circ r^{(2)}_{\psi_{|t|}(F_j)} \circ \diag(\eta_c^{-1}) \big) = {\det}_{{\mathcal N}(\Gamma)} \big( r^{(2)}_{\psi_{|t|}(F_j)} \big) = \Delta^{(2)}_{K,P}(|t|).
\]
This shows $\Delta^{(2)}_{K,P}(t) = \Delta^{(2)}_{K,P}(|t|)$.
(See \cite[Theorem 6.1]{LZ06a} for another proof of this statement for the case $|t| = 1$.)
\end{proof}

The next result --- due to Li and Zhang~\cite{LZ06a} --- ensures that the Fuglede--Kadison determinant does not depend on the choice of the Wirtinger presentation.

\begin{proposition}
Let $P$ and $P'$ be Wirtinger presentations of the knot group of $K$. We denote the associated Fox matrices by $F$ and $F'$ respectively.
If $r^{(2)}_{\psi_t(F_1)}$ is injective, then $r^{(2)}_{\psi_t(F'_1)}$ is injective and there exists $p \in \Z$ such that
\[
\Delta^{(2)}_{K,P}(t) = \Delta^{(2)}_{K,P'}(t) \cdot |t|^p
\]
\end{proposition}
The proof follows by examining the proof of \cite[Proposition 3.4]{LZ06a}.

The proposition above allows us to define the $L^2$-Alexander invariant $\Delta^{(2)}_K$ of the knot $K$.

\begin{definition}[$L^2$-Alexander invariant]\label{def-Alex}
Let $K$ be a knot. Suppose that one (and hence all) Wirtinger presentation $P$ of the knot group of $K$ has the property that for the associated Fox matrix $F$ and all $t \in \C^*$ the map $r^{(2)}_{\psi_t(F_1)}$ is injective with $\det_{{\mathcal N}(\Gamma)}(r^{(2)}_{\psi_t(F_1)}) > 0$.
Notice that $\{ t \mapsto |t|^p \mid p \in \Z \}$ is a subgroup of the multiplicative group $\map(\C^*,\R^{>0})$.
We define the \emph{$L^2$-Alexander invariant}
\[
\Delta^{(2)}_K \in \map(\C^*,\R^{>0}) / \{ t \mapsto |t|^p \mid p \in \Z \}
\]
by $t \mapsto \Delta^{(2)}_{K,P}(t)$.
\end{definition}

Notice that Proposition \ref{prop_|t|} implies $\Delta^{(2)}_K(t) = \Delta^{(2)}_K(|t|)$, where $\Delta^{(2)}_K(t)$ and $\Delta^{(2)}_K(|t|)$ are considered as elements in $\map(\C^*,\R^{>0}) / \{ t \mapsto |t|^p \mid p \in \Z \}$.

For some knots (e.g. the trefoil knot) there exist such simple Wirtinger presentations that one can directly calculate the $L^2$-Alexander invariant from the definition. But in general it is difficult to determine the $L^2$-Alexander invariant.
We are mostly interested in torus knots. The knot group of the torus knot of type $(p,q)$ admits the very simple well--known presentation with two generators and a single relation: $P' = \{ x, y \, | \, {x}^p = {y}^q \}$ (see~\cite{BurdeZieschang}). But unfortunately, this is not a Wirtinger presentation. Nevertheless, the following result gives us a method to compute the Fuglede--Kadison determinant.

\begin{theorem}\label{mainthm}
Let $K$ be a knot in $S^3$ whose knot group is denoted as $\Gamma$. Let $P$ be a Wirtinger presentation with associated maps $\phi \colon \Gamma \to \Z$ and $\psi_t \colon \C \Gamma \to \C \Gamma$.
Let $P'$ be a further default 1 presentation of $\Gamma$ (not necessarily a Wirtinger presentation) with associated Fox matrix $F'$. Suppose there exists $j$ such that $r^{(2)}_{\psi_t(F'_j)}$ is injective and $\det_{{\mathcal N}(\Gamma)}(r^{(2)}_{\psi_t(F'_j)}) > 0$ for all $t \in \C^*$.
Then the Wirtinger presentation $P$ satisfies the assumption of Definition~\ref{def-Alex} and
\begin{equation}\label{EQ:mainthm}
\Delta^{(2)}_K(t) =  {\det}_{{\mathcal N}(\Gamma)}(r^{(2)}_{\psi_{|t|}(F'_j)}) \cdot \max\{ |t| , 1 \}^{1 - \phi(g'_j)}.
\end{equation}
\end{theorem}

The proof is given in Section~\ref{sec_weighted_L2-torsion} (see Corollary~\ref{cor_presentation}). It is based on a systematical study of \emph{weighted $L^2$-invariants} which we introduce in the next section.

\section{Introduction to weighted $L^2$-invariants} \label{sec_weighted_L2-invariants}

In this section we introduce the invariants which we obtain by replacing $l^2(G)$ by another Hilbert space, the weighted Hilbert space $l^2(G,\varrho)$, where $\varrho$ is a representation of $G$ into the multiplicative group $\R^{>0}$. For more information about these invariants we refer to the appendix.

Let $\varrho \colon G \to \R^{>0}$ be a group homomorphism. We define an inner product on the complex group ring $\C G$ by
\[
\langle \sum_{g \in G} c_g \cdot g \, , \, \sum_{g \in G} d_g \cdot g \rangle_\varrho := \sum_{g \in G} c_g \cdot \overline{d_g} \cdot \varrho(g).
\]
The Hilbert space completion with respect to the inner product $\langle \, , \, \rangle_\varrho$ is denoted by $l^2(G,\varrho)$.
The complex group ring $\C G$ together with this inner product and the involution
\[
\big( \sum_{g \in G} c_g \cdot g \big)^* := \sum_{g \in G} c_g \cdot \varrho(g) \cdot g^{-1}
\]
satisfies the axioms of a unital Hilbert algebra, i.e.
\begin{enumerate}
 \item $(c \cdot d)^* = d^* \cdot c^*$,
 \item $\langle c \, , \, d \rangle_\varrho = \langle d^* \, , \, c^* \rangle_\varrho$,
 \item $\langle c \cdot d \, , \, e \rangle_\varrho = \langle d \, , \, c^* \cdot e \rangle_\varrho$,
 \item The map $r_c \colon \C G \to \C G$ given by right multiplication with $c \in \C G$ is continuous.
\end{enumerate}

We define the \emph{weighted von Neumann algebra} ${\mathcal N}(G,\varrho)$ as the algebra of all bounded linear endomorphisms of $l^2(G,\varrho)$ that commute with the left $G$-action.

The \emph{trace} of an element $\phi \in {\mathcal N}(G,\varrho)$ is defined by $\tr_{{\mathcal N}(G,\varrho)}(\phi) := \langle \phi(e) \, , \, e \rangle_\varrho$ where $e \in \C G \subset l^2(G,\varrho)$ denotes the unit element. We can extend this trace to $n \times n$-matrices over ${\mathcal N}(G,\varrho)$ by considering the sum of the traces of the entries on the diagonal.

Obviously, for $\varrho$ the trivial group homomorphism we have $l^2(G,\varrho) = l^2(G)$ and ${\mathcal N}(G,\varrho) = {\mathcal N}(G)$.
If $\varrho \neq 1$, it is sometimes useful to refer to the well-known case $l^2(G)$ resp. ${\mathcal N}(G)$.
Notice that the map
\[
\Phi_\varrho \colon \C G \to \C G, \sum_{g \in G} c_g \cdot g \mapsto \sum_{g \in G} c_g \cdot \sqrt{\varrho(g)} \cdot g
\]
induces an isometry $\Phi_\varrho \colon l^2(G,\varrho) \to l^2(G)$ and an isomorphism $\Phi_\varrho \colon {\mathcal N}(G,\varrho) \to {\mathcal N}(G)$.
We have $\tr_{{\mathcal N}(G,\varrho)}(\phi) = \tr_{{\mathcal N}(G)}(\Phi_\varrho(\phi))$.

\begin{definition}[Hilbert ${\mathcal N}(G,\varrho)$-module]
A \emph{finitely generated Hilbert ${\mathcal N}(G,\varrho)$-module} $V$ is a Hilbert space with a linear left $G$-action such that there exists a $\C G$-linear embedding of the Hilbert space into an orthogonal direct sum of a finite number of copies of $l^2(G,\varrho)$.
\end{definition}

\begin{definition}[von Neumann dimension]
Let $V$ be a finitely generated Hilbert ${\mathcal N}(G,\varrho)$-module.
Choose any orthogonal $G$-equivariant projection $\pr \colon l^2(G,\varrho)^n \to l^2(G,\varrho)^n$ whose image is isometrically $G$-isomorphic to $V$.
The \emph{von Neumann dimension} of $V$ is defined by
\[
\dim_{{\mathcal N}(G,\varrho)}(V) := \tr_{{\mathcal N}(G,\varrho)}(\pr) \in [0,\infty).
\]
\end{definition}

Let $X$ be a finite connected CW-complex with fundamental group $G$. We obtain a Hilbert ${\mathcal N}(G,\varrho)$-chain complex $C^{(2)}_{\varrho,*}({\tilde X}) := l^2(G,\varrho) \otimes_{\Z G} C^{cell}_*({\tilde X})$, where $C^{cell}_*({\tilde X})$ is the cellular chain complex of the universal covering of $X$.

\begin{definition}[weighted $L^2$-homology, weighted $L^2$-Betti number]
We define the \emph{$n$-th (reduced) $L^2$-homology} and the \emph{$n$-th $L^2$-Betti number} of a finite connected CW-complex $X$ by
\begin{eqnarray*}
H_{\varrho,n}^{(2)}({\tilde X}) &:=& \ker(c_{\varrho,n}^{(2)}({\tilde X})) / \overline{\im(c_{\varrho,n+1}^{(2)}({\tilde X}))},\\
b_{\varrho,n}^{(2)}({\tilde X}) &:=& \dim_{{\mathcal N}(G)}(H_{\varrho,n}^{(2)}({\tilde X})).
\end{eqnarray*}
\end{definition}

The authors conjecture that the weighted $L^2$-Betti numbers do not depend on the group homomorphism $\varrho$ (see Conjecture \ref{conj-Betti}). A proof for some groups $G$ is given in Proposition \ref{prop-Betti-independant}.

We are primarily interested in the weighted $L^2$-torsion. The definition is based on the weighted version of the Fuglede--Kadison determinant.

\begin{definition}[Fuglede--Kadison determinant]
Let $f\colon U \to V$ be a morphism of finitely generated Hilbert ${\mathcal N}(G,\varrho)$-modules, i.e. a bounded $G$-equivariant operator. Consider the spectral density function
\[
F(f)\colon \R \to [0,\infty), \lambda \mapsto \dim_{{\mathcal N}(G)}(\im(E_{\lambda^2}^{f^*f}))
\]
where $\{E_\lambda^{f^*f}\colon U \to U \mid \lambda \in \R\}$ denotes the family of spectral projections of the positive endomorphism $f^*f$.
The spectral density function is monotonous and right-continuous. It defines a measure on the Borel $\sigma$-algebra on $\R$ which is uniquely determined by
\[
dF(f)((a,b]) := F(f)(b)-F(f)(a) \mbox{ for } a<b.
\]
We define the \emph{Fuglede--Kadison determinant} of $f$ by
\[
{\textstyle\det}_{{\mathcal N}(G,\varrho)}(f) := \exp(\int_{0^+}^\infty \ln(\lambda) \, dF(f)(\lambda))
\]
if $\int_{0^+}^\infty \ln(\lambda) \, dF(f)(\lambda) > -\infty$ and by $\det_{{\mathcal N}(G)}(f) := 0$ otherwise.
\end{definition}

Let $X$ be a finite CW-complex with fundamental group $G$. In analogy to the classical $L^2$-torsion we would like to define its weighted $L^2$-torsion by
\[
- \sum_{n \geq 0} (-1)^n \cdot \ln({\textstyle\det}_{{\mathcal N}(G,\varrho)}(c_{\varrho,n}^{(2)}({\tilde X}))) \in \R.
\]
But this is not well-defined because it depends on the choice of a cellular basis for $C^{cell}_*({\tilde X})$.

\begin{definition}[weighted $L^2$-torsion]
Let $X$ be a finite CW-complex with fundamental group $G$ and let $\varrho \colon G \to \R^{>0}$ be a group homomorphism. For $x \in \R$ we set $\varrho^x \colon G \to \R^{>0}, g \mapsto \varrho(g)^x$.
Suppose that ${\tilde X}$ is det-$L^2$-acyclic with respect to $\varrho$, i.e. $b^{(2)}_{\varrho^x,n}({\tilde X}) = 0$ and $\det_{{\mathcal N}(G,\varrho^x)}(c^{(2)}_n(X)) > 0$ for all $n \in \Z$ and $x \in \R$.
We define the \emph{weighted $L^2$-torsion}
\[
\rho^{(2)}_\varrho({\tilde X}) \in \map(\R,\R) / \{ x \mapsto \frac{x}{2} \cdot \ln(\varrho(g)) \mid g \in G \}
\]
by $\rho^{(2)}_\varrho({\tilde X})(x) := - \displaystyle\sum_{n \geq 0} (-1)^n \cdot \ln\big({\textstyle\det}_{{\mathcal N}(G,\varrho^x)}(c_{\varrho^x,n}^{(2)}({\tilde X}))\big)$.
\end{definition}

The basic properties of the weighted $L^2$-torsion are listed in the appendix (see Proposition \ref{properties-torsion}).

In the next section we will see that the $L^2$-Alexander polynomial is determined by the weighted $L^2$-torsion of the universal covering of the knot complement. 

\section{Weighted $L^2$-torsion of knot complements} \label{sec_weighted_L2-torsion}

In this section we study the relation between the $L^2$-Alexander polynomial and the weighted $L^2$-torsion of the universal covering of the knot complement.

\begin{proposition} \label{prop_Alex-weighted}
Let $K$ be a knot and $P = \langle g_1, \ldots, g_k \, | \, r_1, \ldots, r_{k-1} \rangle$ a default 1 presentation of its knot group $\Gamma$ (not necessary a Wirtinger presentation).
Let $j \in \{1, \ldots, k\}$ and let $F_j$ be the matrix obtained from the Fox matrix $F = (\partial r_i / \partial g_l)$ by removing the $j$th column.
The universal covering $\widetilde{M_K}$ of the knot complement is det-$L^2$-acyclic w.r.t. $\varrho \colon \Gamma \to \R^{>0}$ if and only if $r^{(2)}_{F_j} \colon l^2(\Gamma,\varrho^x)^{k-1} \to l^2(\Gamma,\varrho^x)^{k-1}$ is injective and $\det_{{\mathcal N}(\Gamma,\varrho^x)}(r^{(2)}_{F_j}) > 0$ for all $x \in \R$.
In this case we have
\[
\rho^{(2)}_\varrho(\widetilde{M_K})(x) = - \ln {\det}_{{\mathcal N}(\Gamma,\varrho^x)}(r^{(2)}_{F_j}) + \max\{ \frac{x}{2} \cdot \ln(\varrho(g_j)) , 0 \}.
\]
\end{proposition}
\begin{proof}
In the proof of \cite[Theorem 2.4]{Luc94} L\"uck considers a $2$-dimensional CW-complex $X$ which is homotopy equivalent to the knot complement $M_K$.
This CW-complex $X$ with fundamental group $\Gamma$ has one cell of dimension zero, $k$ cells of dimension one and $k-1$ cells of dimension two associated to the given presentation $P$.
The cellular chain complex $C_*(\widetilde{X})$ of the universal covering $\widetilde{X}$ looks like
\[
\xymatrix@1@+2.5pc{0 \ar[r] & \oplus_{i=1}^{k-1} \Z \Gamma \ar[r]^-{r_{F}}&  \oplus_{i=1}^k \Z\Gamma \ar[r]^-{\oplus_{i=1}^k r_{g_i - 1}}\ar[r] & \Z \Gamma.}
\]
As in the proof of \cite[Theorem 2.4]{Luc94} we consider the short exact sequence of $\Z \Gamma$-chain complexes $\xymatrix@1@-1pc{0 \ar[r] & C_* \ar[r] & C_*(\widetilde{X}) \ar[r] & D_* \ar[r] & 0}$ where $C_* \colon \Z \Gamma \stackrel{r_{g_j - 1}}{\longrightarrow} \Z \Gamma$ is concentrated in dimensions $0$, $1$ and $D_* \colon \Z \Gamma^{k-1} \stackrel{r_{F_j}}{\longrightarrow} \Z \Gamma^{k-1}$ is concentrated in dimensions $1$, $2$.
Notice that $r^{(2)}_{g_j - 1} \colon l^2(\Gamma,\varrho^x) \to l^2(\Gamma,\varrho^x)$ is injective with
\begin{align*}
& {\det}_{{\mathcal N}(\Gamma,\varrho^x)}(r^{(2)}_{g_j - 1}) = {\det}_{{\mathcal N}(\Gamma)}(\Psi_{\varrho^x}(r^{(2)}_{g_j - 1})) = {\det}_{{\mathcal N}(\Gamma)}(r^{(2)}_{\Phi_{\varrho^x}(g_j - 1)}) = \\
& {\det}_{{\mathcal N}(\Gamma)}(r^{(2)}_{\sqrt{\varrho^x(g_j)} \cdot g_j - 1}) = \max\{ \varrho(g_j)^{x/2} ,  1\}.
\end{align*}
Tensoring the exact sequence $\xymatrix@1@-1pc{0 \ar[r] & C_* \ar[r] & C_*(\widetilde{X}) \ar[r] & D_* \ar[r] & 0}$ with $l^2(\Gamma,\varrho^x)$ and applying Proposition \ref{prop_torsion} (\ref{prop_torsion-1}) shows that $\widetilde{X}$ is det-$L^2$-acyclic w.r.t. $\varrho$ if and only if $r^{(2)}_{F_j} \colon l^2(\Gamma,\varrho^x)^{k-1} \to l^2(\Gamma,\varrho^x)^{k-1}$ is injective and $\det_{{\mathcal N}(\Gamma,\varrho^x)}(r^{(2)}_{F_j}) > 0$ for all $x \in \R$. In this case, we obtain
\begin{eqnarray*}
\rho^{(2)}_\varrho(\widetilde{X})(x) & = & - \ln {\det}_{{\mathcal N}(\Gamma,\varrho^x)}(r^{(2)}_{F_j}) + \ln {\det}_{{\mathcal N}(\Gamma,\varrho^x)}(r^{(2)}_{g_j - 1}) \\
& = & - \ln {\det}_{{\mathcal N}(\Gamma,\varrho^x)}(r^{(2)}_{F_j}) + \max\{ \frac{x}{2} \cdot \ln(\varrho(g_j)) , 0 \}.
\end{eqnarray*}
Since the Whitehead group of a knot group is trivial (see \cite{Wal78}), $X$ is simply homotopy equivalent to $M_K$. We can finally apply Proposition \ref{properties-torsion} (\ref{properties-torsion-1}).
\end{proof}

\begin{corollary} \label{cor_Alex-torsion}
Let $P$ be a Wirtinger presentation of the knot group $\Gamma$ of the knot $K$. We define $\varrho \colon \Gamma \to \R^{>0}$ by $\varrho(g) := \exp(2 \cdot \phi(g))$ where $\phi \colon \Gamma \to \Z$ is the group homomorphism associated to the Wirtinger presentation $P$ sending each generator to $1$.
Suppose that the universal covering $\widetilde{M_K}$ of the knot complement is det-$L^2$-acyclic w.r.t. $\varrho$.
Then the assumption of Definition \ref{def-Alex} is satisfied and
\[
\Delta^{(2)}_K(t) = \Delta^{(2)}_K(|t|) = \exp\big( -\rho^{(2)}_\varrho(\widetilde{M_K})(\ln(|t|)) \big) \cdot \max\{ |t| , 1 \}
\]
\end{corollary}
\begin{proof}
This is a consequence of Proposition \ref{prop_Alex-weighted}. We set $x := \ln(|t|)$.
Notice that
\[
{\det}_{{\mathcal N}(\Gamma,\varrho^x)}(r^{(2)}_{F_1}) = {\det}_{{\mathcal N}(\Gamma)}(\Psi_{\varrho^x}(r^{(2)}_{F_1})) = {\det}_{{\mathcal N}(\Gamma)}(r^{(2)}_{\Phi_{\varrho^x}(F_1)}) = {\det}_{{\mathcal N}(\Gamma)}(r^{(2)}_{\psi_{|t|}(F_1)}).
\]
\end{proof}

\begin{remark}
In terms of the simplified invariants we obtain
\[
{\tilde \Delta}^{(2)}_K(t) = \exp\big( -{\tilde \rho}^{(2)}_\varrho(\widetilde{M_K}) \big) \cdot \max\{ |t| , |t|^{-1} \}
\]
with $\varrho(g) := |t|^{2 \cdot \phi(g)}$.
In \cite[Theorem 3.2]{LZ06b} Li and Zhang study the simplified $L^2$-Alexander invariant
\[
{\Delta'}^{(2)}_K(t) := \sqrt{\frac{\Delta^{(2)}_K(t)}{\max\{ |t| , 1 \}} \cdot \frac{\Delta^{(2)}_K(t^{-1})}{\max\{ |t|^{-1} , 1 \}}}.
\]
From the corollary above we conclude ${\tilde \Delta}^{(2)}_K(t) = \exp( -{\tilde \rho}^{(2)}_\varrho(\widetilde{M_K}))$ with $\varrho(g) := |t|^{2 \cdot \phi(g)}$.
\end{remark}

\begin{corollary} \label{cor_presentation}
Let $K$ be a knot in $S^3$ whose knot group is denoted as $\Gamma$. Let $P$ be a Wirtinger presentation with associated maps $\phi \colon \Gamma \to \Z$ and $\psi_t \colon \C \Gamma \to \C \Gamma$.
Let $P'$ be a further default 1 presentation of $\Gamma$ (not necessarily a Wirtinger presentation) with associated Fox matrix $F'$. Suppose there exists $j$ such that $r^{(2)}_{\psi_t(F'_j)}$ is injective and $\det_{{\mathcal N}(\Gamma)}(r^{(2)}_{\psi_t(F'_j)}) > 0$ for all $t \in \C^*$.
Then the Wirtinger presentation $P$ satisfies the assumption of Definition \ref{def-Alex} and
\[
\Delta^{(2)}_K(t) =  {\det}_{{\mathcal N}(\Gamma)}(r^{(2)}_{\psi_{|t|}(F'_j)}) \cdot \max\{ |t| , 1 \}^{1 - \phi(g'_j)}.
\]
\end{corollary}
\begin{proof}
We conclude from Proposition \ref{prop_Alex-weighted} that the universal covering $\widetilde{M_K}$ of the knot complement is det-$L^2$-acyclic w.r.t. $\varrho \colon \Gamma \to \R^{>0}, g \mapsto \exp(2 \cdot \phi(g))$.
Setting $x := \ln(|t|)$ yields
\[
{\det}_{{\mathcal N}(\Gamma)}(r^{(2)}_{\psi_{|t|}(F'_1)}) = \exp\big( -\rho^{(2)}_\varrho(\widetilde{M_K})(\ln(|t|)) \big) \cdot \max\{ |t| , 1 \}^{\phi(g'_j)}.
\]
Now the statement follows using Corollary \ref{cor_Alex-torsion}.
\end{proof}

\begin{proposition} \label{prop_pieces}
The knot complement $M_K$ of a non-trivial knot $K$ is an irreducible compact connected oriented $3$-manifold whose boundary is an incompressible torus. There is a geometric toral splitting of $M_K$ along disjoint incompressible $2$-sided tori in $M_K$ whose pieces are Seifert manifolds or hyperbolic manifolds. Let $S_1, S_2, \ldots, S_q$ be the Seifert pieces and $H_1, H_2, \ldots, H_r$ the hyperbolic pieces.
Let $\varrho \colon \pi_1(M_K) \to \R^{>0}$ be a group homomorphism. Suppose that the universal coverings of the hyperbolic pieces are det-$L^2$-acyclic w.r.t. $\varrho_{H_k} \colon \pi_1(H_k) \to \pi_1(M_K) \stackrel{\varrho}{\longrightarrow} \R^{>0}$. Then $M_K$ is det-$L^2$-acyclic w.r.t. $\varrho$ and
\[
\rho^{(2)}_\varrho(\widetilde{M_K})(x) = \sum_{k=1}^q \rho^{(2)}_{\varrho_{S_k}}(\widetilde{S_k})(x) + \sum_{k=1}^r \rho^{(2)}_{\varrho_{H_k}}(\widetilde{H_k})(x).
\]
Moreover, for every $k \in \{1, \ldots, q\}$ there exist $d \in \N$ and $g \in \pi_1(S_k)$ such that
\[
d \cdot \rho^{(2)}_{\varrho_{S_k}}(\widetilde{S_k})(x) = \max\{\frac{x}{2} \cdot \ln(\varrho_{S_k}(g)) , 0\}.
\]
\end{proposition}
\begin{proof}
The proof is similar to the proof of \cite[Theorem 4.6 resp. Theorem 4.3]{Luc02}.
A Seifert manifold admits a finite covering which is the total space of a $S^1$-principal bundle over a compact orientable surface. We conclude from Proposition \ref{properties-torsion} (\ref{properties-torsion-4}), (\ref{properties-torsion-7}) that $S_k$ is det-$L^2$-acyclic w.r.t. $\varrho_{S_k}$ and that
\[
d \cdot \rho^{(2)}_{\varrho_{S_k}}(\widetilde{S_k})(x) = \max\{\frac{x}{2} \cdot \ln(\varrho_{S_k}(g)) , 0\}
\]
holds for some $d \in \N$, $g \in \pi_1(S_k)$.
The formula
\[
\rho^{(2)}_\varrho(\widetilde{M_K})(x) = \sum_{k=1}^q \rho^{(2)}_{\varrho_{S_k}}(\widetilde{S_k})(x) + \sum_{k=1}^r \rho^{(2)}_{\varrho_{H_k}}(\widetilde{H_k})(x)
\]
follows from Proposition \ref{properties-torsion} (\ref{properties-torsion-2}) resp. Remark \ref{rem_sum-formula}.
\end{proof}

\begin{remark} \label{rem_pieces}
With some more effort we can prove that
\[
\sum_{k=1}^q \rho^{(2)}_{\varrho_{S_k}}(\widetilde{S_k})(x) = \max\{\frac{x}{2} \cdot \ln(\varrho(g)) , 0\}
\]
holds for some $g \in \pi_1(M_K)$.
For the proof we consider
\[
[\rho^{(2)}_{\varrho_{S_k}}(\widetilde{S_k})] \in \map(\R,\R) / \{ x \mapsto \max\{ \frac{x}{2} \cdot \ln(\varrho(g_1)) , \frac{x}{2} \cdot \ln(\varrho(g_2)) \} \mid g_1, g_2 \in \pi_1(M_K) \}
\]
and have to show $[\rho^{(2)}_{\varrho_{S_k}}(\widetilde{S_k})] = 0$.
By \cite[Corollary 2]{Gab91} we know that $\pi_1(S_k)$ contains a normal infinite cyclic subgroup.
Now, we can use the same arguments as in the proof of \cite[Theorem 1 (3)]{Weg09} (resp. \cite[Theorem 3.113]{Luc02}) to prove $[\rho^{(2)}_{\varrho_{S_k}}(\widetilde{S_k})] = 0$.\\
In particular, if $M_K$ contains no hyperbolic pieces then we have
\[
\rho^{(2)}_\varrho(\widetilde{M_K})(x) = \max\{\frac{x}{2} \cdot \ln(\varrho(g)) , 0\}
\]
for some $g \in \pi_1(M_K)$.
This implies $\Delta^{(2)}_K(t) = \max\{|t|,1\}^n$ for some $n \in \N$.\\
In general, we obtain the formula
\[
\Delta^{(2)}_K(t) = \max\{|t|,1\}^n \cdot \exp\big(- \sum_{k=1}^r \rho^{(2)}_{\varrho_{H_k}}(\widetilde{H_k})(\ln(|t|))\big)
\]
for some $n \in \N$.
\end{remark} 

\section{Computation for torus knots}
\label{S:computations}

Let $(p,q)$ be a pair of coprime integers. We let $T(p,q)$ denote the torus knot of type $(p,q)$.

\begin{proposition}
Consider the trefoil knot $K$ with the Wirtinger presentation
\[
P = \big\langle a, b \, \big| \, a b a = b a b \big\rangle
\]
of the associated knot group $\Gamma$. One has
\[
\Delta_{K,P}^{(2)}(t) = \max\{|t|,1\}^2.
\]
\end{proposition}
\begin{proof}
We have $\psi_t(F_2) = ( 1 - t b + t^2 a b )$ and hence $\Delta_{K,P}^{(2)}(t) = {\det}_\Gamma ( r^{(2)}_{1 - t b + t^2 a b} )$.\\
For $n \in \Z$ we set
\[
d_n := \sum_{k=0}^{\lfloor n/3 \rfloor} (-1)^k \cdot b^{n-3k} \cdot (bab)^k \in \C \Gamma.
\]
Notice that $d_n - b \cdot d_{n-1}$ is non-zero if and only if $n$ is a non-negative multiple of $3$.
In this case we have
\begin{equation} \label{eq}
d_n - b \cdot d_{n-1} = (-1)^{n/3} \cdot (bab)^{n/3}.
\end{equation}
For $|t| < 1$ the element
\[
c_t := \sum_{n=0}^\infty \big( d_n - ab \cdot d_{n-2} \big) \cdot t^n
\]
lies in $l^1(\Gamma)$ because of
\begin{eqnarray*}
\Big\| \sum_{n=0}^\infty \big( d_n - ab \cdot d_{n-2} \big) \cdot t^n \Big\|_{l^1} & \leq & \sum_{n=0}^\infty \big( \lfloor \frac{n+3}{3} \rfloor + \lfloor \frac{n+1}{3} \rfloor \big) \cdot |t|^n \\
& \leq & \sum_{n=0}^\infty n \cdot |t|^n \; = \; \frac{|t|}{(1-|t|)^2} \; < \; \infty.
\end{eqnarray*}
An easy calculation shows $(1 - t b + t^2 a b) \cdot c_t = 1$. (Hint: Use equation (\ref{eq}).)
We conclude that $r^{(2)}_{1 - t b + t^2 ab}$ is invertible in ${\mathcal N}(\Gamma)$ with inverse $r^{(2)}_{c_t}$.
We define
\[
f_{t,u} := r^{(2)}_{1 - (tu) b + (tu)^2 ab} \; \mbox{for} \; 0 \leq u \leq 1.
\]
Observe that \cite[Theorem 1.10(e)]{CFM97} tells us
\[
\Delta_{K,P}^{(2)}(t) = {\det}_\Gamma (f_{t,1}) = \exp \Big( -\Re \Big( \int_0^1 \tr_{{\mathcal N}(\Gamma)} \big( \sum_{k=0}^\infty r^{(2)}_{c_{tu}} \circ r^{(2)}_{- t b + 2 t^2 u ab} \big) \d u \Big) \Big).
\]
We have
\[
r^{(2)}_{c_{tu}} \circ r^{(2)}_{- t b + 2 t^2 u ab} = r^{(2)}_{(- t b + 2 t^2 u ab) \cdot c_{tu}} =
\sum_{n=0}^\infty r^{(2)}_{(- t b + 2 t^2 u ab) \cdot \sum_{n=0}^\infty ( d_n - ab \cdot d_{n-2} ) \cdot (tu)^n}.
\]
Under the group homomorphism $\Gamma \to \Z, a, b \mapsto 1$, every element in the support of $d_n - ab \cdot d_{n-2}$ is mapped to $n \in \N^*$. This shows that every element in the support of $(- t b + 2 t^2 u ab) \cdot \sum_{n=0}^\infty ( d_n - ab \cdot d_{n-2} ) \cdot (tu)^n$ is mapped to a positive integer; in particular, the unit element $e \in \Gamma$ does not lie in the support. Hence
\[
\tr_{{\mathcal N}(\Gamma)} \Big( r^{(2)}_{(- t b + 2 t^2 u ab) \cdot \sum_{n=0}^\infty ( d_n - ab \cdot d_{n-2} ) \cdot (tu)^n} \Big) = 0.
\]
This shows $\Delta_{K,P}^{(2)}(t) = \exp(0) = 1$ for all $0 < |t| < 1$.\\
In the case $|t| > 1$ we can argue analogously. We obtain
\[
{\det}_{{\mathcal N}(\Gamma)} \big( r^{(2)}_{1 - t^{-1} a^{-1} + t^{-2} b^{-1}a^{-1}} \big) = 1
\]
which implies
\[
\Delta_{K,P}^{(2)}(t) = {\det}_{{\mathcal N}(\Gamma)} ( r^{(2)}_{1 - t b + t^2 ab} ) = {\det}_{{\mathcal N}(\Gamma)} \big( r^{(2)}_{1 - t^{-1} a^{-1} + t^{-2} b^{-1}a^{-1}} \big) \cdot {\det}_{{\mathcal N}(\Gamma)} \big( r^{(2)}_{t^2 ab} \big) = |t|^2.
\]
For $|t| = 1$ we have $\Delta_{K,P}^{(2)}(t) = \exp\big( -\rho^{(2)}(\widetilde{M_K}) \big)$.
It is a known fact that $\rho^{(2)}(\widetilde{M_K}) = 0$.
\end{proof}

Using Corollary \ref{cor_presentation} we can calculate the $L^2$-Alexander invariant for all torus knots.
\begin{proposition}
The $L^2$-Alexander invariant of the torus knot $T(m,n)$ is
\[
\Delta^{(2)}_{T(m,n)}(t) = \max\{|t|,1\}^{(m-1)(n-1)}.
\]
\end{proposition}
\begin{proof}
The knot group of $T(m,n)$ has the presentation $P' = \{ g'_1, g'_2 \, | \, {g'_1}^n = {g'_2}^m \}$.
Since the natural image of $g'_1$ (resp. $g'_2$) in $H_1(S^3 - T(m,n))$ is $h^m$ (resp. $h^n$), we obtain $\phi(g'_1)=m$ (resp. $\phi(g'_2)=n$).
The entry of the matrix $F'_2$ is
\[
\frac{\partial r}{\partial g'_1} = {g'_1}^{n-1} + {g'_1}^{n-2} + \cdots + g'_1 + 1.
\]
We conclude
\[
{\det}_{{\mathcal N}(\Gamma)}(r^{(2)}_{\psi_{|t|}(F'_2)}) \cdot {\det}_{{\mathcal N}(\Gamma)}(r^{(2)}_{\psi_{|t|}(g'_1-1)}) = {\det}_{{\mathcal N}(\Gamma)}(r^{(2)}_{\psi_{|t|}({g'_1}^n-1)})
\]
with
\begin{align*}
& {\det}_{{\mathcal N}(\Gamma)}(r^{(2)}_{\psi_{|t|}(g'_1-1)}) = {\det}_{{\mathcal N}(\Gamma)}(r^{(2)}_{|t|^m \cdot g'_1 - 1}) = \max\{ |t|^m , 1 \} = \max\{ |t| , 1 \}^m, \\
& {\det}_{{\mathcal N}(\Gamma)}(r^{(2)}_{\psi_{|t|}({g'_1}^n-1)}) = {\det}_{{\mathcal N}(\Gamma)}(r^{(2)}_{|t|^{mn} \cdot {g'_1}^n - 1}) = \max\{ |t|^{mn} , 1 \} = \max\{ |t| , 1 \}^{mn}.
\end{align*}
This shows ${\det}_{{\mathcal N}(\Gamma)}(r^{(2)}_{\psi_{|t|}(F'_2)}) = \max\{ |t| , 1 \}^{mn-m}$.
Using Corollary \ref{cor_presentation} we obtain
\begin{eqnarray*}
\Delta^{(2)}_K(t) & = & {\det}_{{\mathcal N}(\Gamma)}(r^{(2)}_{\psi_{|t|}(F'_j)}) \cdot \max\{ |t| , 1 \}^{1 - \phi(g'_2)} \\
& = & \max\{ |t| , 1 \}^{mn-m} \cdot \max\{ |t| , 1 \}^{1-n} \\
& = & \max\{ |t| , 1 \}^{(m-1)(n-1)}.
\end{eqnarray*}
\end{proof}

\begin{remark}
Observe that the power $(m-1)(n-1)$ which appears in the formula of the $L^2$-Alexander invariant of the torus knot $T(m,n)$ is equal to $2 \cdot \mathrm{genus}(T(m,n))$. This property will be discussed in the case of fibered knots in a forthcoming paper by Dubois and Friedl~\cite{DubFried}.
\end{remark} 

\section*{Appendix: Weighted $L^2$-invariants}
\addtocounter{section}{1}

In this section, we study the weighted $L^2$-invariants in details. For a short summary see section \ref{sec_weighted_L2-invariants}.

Let $\varrho \colon G \to \R^{>0}$ be a group homomorphism. Recall that we define an inner product on the complex group ring $\C G$ by setting
\[
\langle \sum_{g \in G} c_g \cdot g \, , \, \sum_{g \in G} d_g \cdot g \rangle_\varrho := \sum_{g \in G} c_g \cdot \overline{d_g} \cdot \varrho(g).
\]
As we have already observed, the Hilbert space completion with respect to the inner product $\langle \, , \, \rangle_\varrho$ is denoted by $l^2(G,\varrho)$.
The complex group ring $\C G$ together with this inner product and the involution
\[
\big( \sum_{g \in G} c_g \cdot g \big)^* := \sum_{g \in G} c_g \cdot \varrho(g) \cdot g^{-1}
\]
satisfies the axioms of a unital Hilbert algebra, i.e.
\begin{enumerate}
 \item $(c \cdot d)^* = d^* \cdot c^*$,
 \item $\langle c \, , \, d \rangle_\varrho = \langle d^* \, , \, c^* \rangle_\varrho$,
 \item $\langle c \cdot d \, , \, e \rangle_\varrho = \langle d \, , \, c^* \cdot e \rangle_\varrho$,
 \item The map $r_c \colon \C G \to \C G$ given by right multiplication with $c \in \C G$ is continuous.
\end{enumerate}
There are three equivalent definitions of the von Neumann algebra ${\mathcal N}(G,\varrho)$:
\begin{enumerate}
\item ${\mathcal N}(G,\varrho)$ is the algebra of all bounded linear endomorphisms of $l^2(G,\varrho)$ that commute with the left $\C G$-action.
\item ${\mathcal N}(G,\varrho)$ is the double commutant of the right $\C G$-action on $l^2(G,\varrho)$.
\item ${\mathcal N}(G,\varrho)$ is the weak closure of $\C G$ acting from the right on $l^2(G,\varrho)$.
\end{enumerate}
We usually refer to the first definition.

The trace of an element $\phi \in {\mathcal N}(G,\varrho)$ is defined by $\tr_{{\mathcal N}(G,\varrho)}(\phi) := \langle \phi(e) \, , \, e \rangle_\varrho$ where $e \in \C G \subset l^2(G,\varrho)$ denotes the unit element. We can extend this trace to $n \times n$-matrices over ${\mathcal N}(G,\varrho)$ by considering the sum of the traces of the entries on the diagonal.

Obviously, for $\varrho$ the trivial group homomorphism we have $l^2(G,\varrho) = l^2(G)$ and ${\mathcal N}(G,\varrho) = {\mathcal N}(G)$.
If $\varrho \neq 1$, it is sometimes useful to refer to the well-known case $l^2(G)$ resp. ${\mathcal N}(G)$.
Notice that the map
\[
\Phi_\varrho \colon \C G \to \C G, \sum_{g \in G} c_g \cdot g \mapsto \sum_{g \in G} c_g \cdot \sqrt{\varrho(g)} \cdot g
\]
induces an isometry $\Phi_\varrho \colon l^2(G,\varrho) \to l^2(G)$ and an isomorphism $\Phi_\varrho \colon {\mathcal N}(G,\varrho) \to {\mathcal N}(G)$.
We have $\tr_{{\mathcal N}(G,\varrho)}(\phi) = \tr_{{\mathcal N}(G)}(\Phi_\varrho(\phi))$.

\subsection{Hilbert ${\mathcal N}(G,\varrho)$-modules}

A finitely generated Hilbert ${\mathcal N}(G,\varrho)$-module $V$ is a Hilbert space $V$ with a linear left $G$-action such that there exists a $\C G$-linear embedding of $V$ into an orthogonal direct sum of a finite number of copies of $l^2(G,\varrho)$. The existence of the embedding implies $\langle g \cdot v \, , \, g \cdot w \rangle = \varrho(g) \cdot \langle v \, , \, w \rangle$ for all $g \in G$, $v,w \in V$.
A map of Hilbert ${\mathcal N}(G,\varrho)$-modules is a bounded $\C G$-linear map. It is weakly surjective if it has dense image, it is a weak isomorphism if it is injective and weakly surjective.

The von Neumann dimension of a finitely generated Hilbert ${\mathcal N}(G,\varrho)$-module $V$ is defined by $\dim_{{\mathcal N}(G,\varrho)}(V) := \tr_{{\mathcal N}(G,\varrho)}(\pr_V) \in \R^{>0}$. Here $\pr_V \colon \oplus_{i=1}^k l^2(G,\varrho) \to \oplus_{i=1}^k l^2(G,\varrho)$ denotes the orthogonal projection onto $V$. The von Neumann dimension $\dim_{{\mathcal N}(G,\varrho)}(V)$ does not depend on the choice of the embedding of $V$ into a finite number of copies of $l^2(G,\varrho)$. We list some properties of the von Neumann dimension:
\begin{enumerate}
\item $\dim_{{\mathcal N}(G,\varrho)}(V) = 0$ if and only if $V=0$.
\item $\dim_{{\mathcal N}(G,\varrho)}(l^2(G,\varrho)) = 1$
\item If $G$ is finite then $\dim_{{\mathcal N}(G,\varrho)}(V) = \dim_\C(V) / |G|$.
\item If $0 \to U \to V \to W \to 0$ is a weakly exact sequence of finitely generated Hilbert ${\mathcal N}(G,\varrho)$-modules then
\[
\dim_{{\mathcal N}(G,\varrho)}(V) = \dim_{{\mathcal N}(G,\varrho)}(U) + \dim_{{\mathcal N}(G,\varrho)}(W).
\] \label{dim_ses}
\end{enumerate}

We can identify the category of finitely generated Hilbert ${\mathcal N}(G,\varrho)$-modules with the category of finitely generated Hilbert ${\mathcal N}(G)$-modules using the functor $\Psi_\varrho$ which we define next. Let $V$ be a finitely generated Hilbert ${\mathcal N}(G,\varrho)$-module. Then $\Psi_\varrho(V)$ is defined as the Hilbert space $V$ with the $G$-action given by $g(v) := \frac{1}{\sqrt{\varrho(g)}} \cdot g \cdot v$. For a morphism $f$ we set $\Psi_\varrho(f) := f$.
We obtain an isometry
\[
l^2(G) \to \Psi_\varrho(l^2(G,\varrho)), \sum_{g \in G} c_g \cdot g \mapsto \sum_{g \in G} \frac{c_g}{\sqrt{\varrho(g)}} \cdot g.
\]
Notice that
\begin{align*}
\dim_{{\mathcal N}(G,\varrho)}(V) & = \tr_{{\mathcal N}(G,\varrho)}(\pr_V) = \tr_{{\mathcal N}(G)}(\Phi_\varrho(\pr_V)) = \\
& = \tr_{{\mathcal N}(G)}(\pr_{\Psi_\varrho(V)}) = \dim_{{\mathcal N}(G)}(\Psi_\varrho(V)).
\end{align*}

\begin{remark}[Restriction]
Let $H < G$ be a subgroup of finite index. We obtain a functor $\res_H$ from the category of finitely generated Hilbert ${\mathcal N}(G,\varrho)$-modules to the category of finitely generated Hilbert ${\mathcal N}(H,\varrho|_H)$-modules by restricting the left $G$-action to an $H$-action.
For a right transversal $T$ of $H$ in $G$ we obtain a $\C H$-linear isometry
\begin{eqnarray*}
\bigoplus_{g \in T} l^2(H,\varrho|_H) & \to & \res_H(l^2(G,\varrho)),\\
\big\{ v_g \, \big| \, g \in T \big\} & \mapsto & \sum_{g \in T} \frac{1}{\sqrt{\varrho(g)}} \cdot v_g \cdot g.
\end{eqnarray*}
A simple calculation shows
\[
\tr_{{\mathcal N}(H,\varrho|_H)}\big( \res_H(\phi) \big) = [G:H] \cdot \tr_{{\mathcal N}(G,\varrho)}(\phi)
\]
for $\phi \colon l^2(G)^n \to l^2(G)^n$.
We conclude
\[
\dim_{{\mathcal N}(H,\varrho|_H)}(\res_H(V)) = [G:H] \cdot \dim_{{\mathcal N}(G,\varrho)}(V).
\]
Notice that $\res_H \circ \Psi_\varrho = \Psi_{\varrho|_H} \circ \res_H$.
\end{remark}

\begin{remark}[Induction]
Let $\varrho \colon G \to \R^{>0}$ be a group homomorphism and $H < G$ a subgroup. We can assign to a finitely generated Hilbert ${\mathcal N}(H,\varrho|_H)$-module $V$ a finitely generated Hilbert ${\mathcal N}(G,\varrho)$-module $\ind_G(V)$ as follows. Let $T$ be a left transversal of $H$ in $G$. There is a pre-Hilbert structure on $\C G \otimes_{\C H} V$ given by
\[
\langle \sum_{g \in T} g \otimes v_g \, , \, \sum_{g \in T} g \otimes w_g \rangle := \sum_{g \in T} \varrho(g) \cdot \langle v_g \, , \, w_g \rangle_{\varrho|_H}.
\]
Notice that this definition does not depend on the choice of the left transversal. The Hilbert space completion of $\C G \otimes_{\C H} V$ is a finitely generated ${\mathcal N}(G,\varrho)$-module which we denote by $\ind_G(V)$.
Notice that $\ind_G(l^2(H,\varrho|_H)) = l^2(G,\varrho)$. Right multiplication with a matrix $A \in M(m \times n; \C H) \subset M(m \times n; \C G)$ defines morphisms $r^{(2)}_A \colon l^2(H,\varrho|_H)^m \to l^2(H,\varrho|_H)^n$ and $r^{(2)}_A \colon l^2(G,\varrho)^m \to l^2(H,\varrho)^n$ such that $\ind_G(r^{(2)}_A) = r^{(2)}_A$.
We calculate
\[
\tr_{{\mathcal N}(G,\varrho)}(\ind_G(\phi)) = \tr_{{\mathcal N}(H,\varrho|_H)}(\phi)
\]
for $\phi \colon l^2(H)^n \to l^2(H)^n$ and conclude
\[
\dim_{{\mathcal N}(G,\varrho)}(\ind_G(V)) = \dim_{{\mathcal N}(H,\varrho|_H)}(V).
\]
We have a $\C G$-linear isometry $\ind_G(\Psi_{\varrho|_H}(V)) \to \Psi_\varrho(\ind_G(V)), c \otimes v \mapsto \Phi_\varrho(c) \otimes v$.
\end{remark}

\subsection{Weighted $L^2$-Betti numbers}

\begin{definition}
Let $(C_*,c_*)$ be a finite Hilbert ${\mathcal N}(G,\varrho)$-chain complex. We define its (reduced) $L^2$-homology and its $L^2$-Betti numbers by
\begin{eqnarray*}
H^{(2)}_{\varrho,n}(C_*) & :=  & \ker(c_n) / \overline{\im(c_{n+1})},\\
b^{(2)}_{\varrho,n}(C_*) & :=  & \dim_{{\mathcal N}(G,\varrho)}\big(H^{(2)}_{\varrho,n}(C_*)\big).
\end{eqnarray*}
\end{definition}

Notice that $\Psi_\varrho(H^{(2)}_{\varrho,n}(C_*)) = H^{(2)}_n(\Psi_\varrho(C_*))$ and $b^{(2)}_{\varrho,n}(C_*) = b^{(2)}_n(\Psi_\varrho(C_*))$.

\begin{proposition} \label{prop_ex-seq}
\begin{enumerate}
\item An exact sequence $0 \to C_* \to D_* \to E_* \to 0$ of finite Hilbert ${\mathcal N}(G,\varrho)$-chain complexes induces a long homology sequence
\[
\cdots \to H^{(2)}_{\varrho,n+1}(E_*) \to H^{(2)}_{\varrho,n}(C_*) \to H^{(2)}_{\varrho,n}(D_*) \to H^{(2)}_{\varrho,n}(E_*) \to H^{(2)}_{\varrho,n-1}(C_*) \to \cdots
\]
which is weakly exact. \label{prop_ex-seq-1}
\item Let $f_* \colon C_* \to D_*$ be a homotopy equivalence of finite Hilbert ${\mathcal N}(G,\varrho)$-chain complexes. Then
\[
b^{(2)}_{\varrho,n}(C_*) = b^{(2)}_{\varrho,n}(D_*).
\] \label{prop_ex-seq-2}
\end{enumerate}
\end{proposition}
\begin{proof}
\begin{enumerate}
\item We can use the functor $\Psi_\varrho$ from the category of finitely generated Hilbert ${\mathcal N}(G,\varrho)$-modules to the category of finitely generated Hilbert ${\mathcal N}(G)$-modules. Hence it suffices to show the statement for finite Hilbert ${\mathcal N}(G)$-chain complexes. This is done in \cite[Theorem 1.21]{Luc02}.
\item Consider the short exact sequence $0 \to D_* \to \cone_*(f_*) \to C_{*-1} \to 0$ and the induced long homology sequence. Since $\cone_*(f_*)$ is contractible, we conclude $H^{(2)}_{\varrho,n}(\cone_*(f_*)) = 0$. We obtain a weak isomorphism $H^{(2)}_{\varrho,n}(C_*) \to H^{(2)}_{\varrho,n}(D_*)$. This implies $b^{(2)}_{\varrho,n}(C_*) = b^{(2)}_{\varrho,n}(D_*)$.
\end{enumerate}
\end{proof}

\begin{definition} \label{def-L2-Betti}
Let $X$ be a finite free $G$-CW-complex and $\varrho \colon G \to \R^{>0}$ a group homomorphism.
Notice that the cellular chain complex $C^{cell}_*(X)$ is a free $\Z G$-chain complex and has a cellular $\Z G$-basis which is unique up to permutation and multiplication with $\pm g \in \Z G$ for $g \in G$. After fixing such a basis we can define the Hilbert ${\mathcal N}(G,\varrho)$-chain complex
\[
C^{(2)}_{\varrho,*}(X) := l^2(G,\varrho) \otimes_{\Z G} C^{cell}_*(X).
\]
We define the weighted $L^2$-Betti numbers of $X$ by
\[
b^{(2)}_{\varrho,n}(X) := b^{(2)}_{\varrho,n}\big(C^{(2)}_{\varrho,*}(X)\big).
\]
\end{definition}

Notice that $b^{(2)}_{\varrho,n}(X)$ does not depend on the choice of the cellular basis. If we take another basis then we obtain an isomorphic Hilbert ${\mathcal N}(G,\varrho)$-chain complex. Proposition \ref{prop_ex-seq} (\ref{prop_ex-seq-2}) implies that the $L^2$-Betti numbers coincide.

Obviously, for the trivial group homomorphism we obtain $b^{(2)}_{1,n}(X) = b^{(2)}_n(X)$.

\begin{example} \label{ex-S1-Betti}
The universal covering $\widetilde{S^1}$ of $S^1$ is a finite $\pi_1(S^1)$-chain complex. We write $\Z \cong \pi_1(S^1) = \langle t \rangle$. Let $\varrho \colon \pi_1(S^1) \to \R^{>0}$ be a group homomorphism.
The cellular chain complex $C^{cell}_*(\widetilde{S^1})$ is given by
\[
\Z \langle t \rangle \stackrel{\cdot (t-1)}{\longrightarrow} \Z \langle t \rangle.
\]
We obtain the following Hilbert ${\mathcal N}(G,\varrho)$-chain complex $C^{(2)}_{\varrho,*}(\widetilde{S^1})$:
\[
l^2(G,\varrho) \stackrel{\cdot (t-1)}{\longrightarrow} l^2(G,\varrho).
\]
An easy calculation shows that the multiplication with $t-1$ is injective. Hence $b^{(2)}_{\varrho,1}(\widetilde{S^1})=0$.
We conclude from the properties of the dimension function that every injective endomorphism of a finitely generated Hilbert ${\mathcal N}(G,\varrho)$-module is a weak isomorphism.
This shows $b^{(2)}_{\varrho,0}(\widetilde{S^1})=0$.
\end{example}

The following proposition is taken from \cite[Theorem 1.35 and Theorem 1.40]{Luc02} where it is stated and proven for the case $\varrho = 1$.
\begin{proposition} \label{properties-Betti}
\begin{enumerate}
\item Homotopy invariance \\ \label{properties-Betti-1}
Let $f \colon X \to Y$  be a $G$-map of finite free $G$-CW-complexes and let $\varrho \colon G \to \R^{>0}$ be a group homomorphism. If the map induced on homology with complex coefficients $H_n(f;\C) \colon H_n(X;\C) \to H_n(Y;\C)$ is bijective for $n < d$ and surjective for $n = d$ then
\begin{eqnarray*}
b^{(2)}_{\varrho,n}(X) & = & b^{(2)}_{\varrho,n}(Y) \quad \textrm{for} \; n < d, \\
b^{(2)}_{\varrho,d}(X) & \geq & b^{(2)}_{\varrho,d}(Y).
\end{eqnarray*}
In particular, if $f$ is a weak homotopy equivalence then
\[
b^{(2)}_{\varrho,n}(X) = b^{(2)}_{\varrho,n}(Y) \quad \textrm{for all} \; n.
\]
\item Euler-Poincar\'e formula \\ \label{properties-Betti-2}
Let $X$ be a finite free $G$-CW-complex. Let $\chi(G \setminus X)$ be the Euler characteristic of the finite CW-complex $G \setminus X$, i.e.
\[
\chi(G \setminus X) := \sum_{n \geq 0} (-1)^n \cdot \beta_n(G \setminus X) \in \Z
\]
where $\beta_n(G \setminus X)$ is the number of $n$-cells of $G \setminus X$.
Then
\[
\chi(G \setminus X) = \sum_{n \geq 0} (-1)^n \cdot b^{(2)}_{\varrho,n}(X).
\]
\item Poincar\'e duality \\ \label{properties-Betti-3}
Let $M$ be a cocompact free proper $G$-manifold of dimension $m$ which is orientable. Then
\[
b^{(2)}_{\varrho,n}(M) = b^{(2)}_{\varrho,m-n}(M, \partial M).
\]
\item Wedge \\ \label{properties-Betti-4}
Let $X_1, X_2$ be finite connected pointed CW-complexes and $X = X_1 \vee X_2$ be their wedge.
Let $\varrho \colon \pi_1(X) = \pi_1(X_1) \ast \pi_1(X_2) \to \R^{>0}$ be a group homomorphism.
Then
\begin{eqnarray*}
b^{(2)}_{\varrho,1}(\widetilde{X}) - b^{(2)}_{\varrho,0}(\widetilde{X}) + 1 & = & \sum_{i=1}^2 \big( b^{(2)}_{\varrho|_{\pi_1(X_i),1}}(\widetilde{X_i}) - b^{(2)}_{\varrho|_{\pi_1(X_i),0}}(\widetilde{X_i}) + 1 \big), \\
b^{(2)}_{\varrho,n}(\widetilde{X}) & = & b^{(2)}_{\varrho|_{\pi_1(X_1),n}}(\widetilde{X_1}) + b^{(2)}_{\varrho|_{\pi_1(X_2),n}}(\widetilde{X_2}) \quad \textrm{for} \; n \geq 2.
\end{eqnarray*}
\item Connected sum \\ \label{properties-Betti-5}
Let $M_1, M_2$ be compact connected $m$-dimensional manifolds with $m \geq 3$ and $M = M_1 \# M_2$ be their connected sum.
Let $\varrho \colon \pi_1(M) \cong \pi_1(M_1) \ast \pi_1(M_2) \to \R^{>0}$ be a group homomorphism.
Then
\begin{eqnarray*}
b^{(2)}_{\varrho,1}(\widetilde{M}) - b^{(2)}_{\varrho,0}(\widetilde{M}) + 1 & = & \sum_{i=1}^2 \big( b^{(2)}_{\varrho|_{\pi_1(M_i),1}}(\widetilde{M_i}) - b^{(2)}_{\varrho|_{\pi_1(M_i),0}}(\widetilde{M_i}) + 1 \big), \\
b^{(2)}_{\varrho,n}(\widetilde{M}) & = & b^{(2)}_{\varrho|_{\pi_1(M_1),n}}(\widetilde{M_1}) + b^{(2)}_{\varrho|_{\pi_1(M_2),n}}(\widetilde{M_2}) \quad \textrm{for} \; n \geq 2.
\end{eqnarray*}
\item Zero-th $L^2$-Betti number \\ \label{properties-Betti-6}
Let $X$ be a connected finite free $G$-CW-complex. Then
\[
b^{(2)}_{\varrho,0}(X) = \frac{1}{|G|}
\]
where $\frac{1}{|G|}$ is to be understood to be zero if the order $|G|$ of $G$ is infinite.
\item Restriction \\ \label{properties-Betti-7}
Let $X$ be a finite free $G$-CW-complex and let $H < G$ be a subgroup of finite index.
Let $\res_H(X)$ be the finite free $H$-CW-complex obtained from $X$ by restricting the $G$-action to an $H$-action.
We have
\[
b^{(2)}_{\varrho|_H,n}(\res_H(X)) = [G:H] \cdot b^{(2)}_{\varrho,n}(X).
\]
\item Induction \\ \label{properties-Betti-8}
Let $X$ be a finite free $H$-CW-complex and let $H < G$ be a subgroup.
Then $G \times_H X$ is a finite free $G$-CW-complex and
\[
b^{(2)}_{\varrho,n}(G \times_H X) = b^{(2)}_{\varrho|_H,n}(X).
\]
\item $S^1$-actions \\ \label{properties-Betti-9}
Let $X$ be a connected $S^1$-CW-complex of finite type, for instance a connected compact manifold with smooth $S^1$-action. Suppose that for one orbit $S^1/H$ (and hence for all orbits) the inclusion into $X$ induces a map on $\pi_1$ with infinite image. (In particular the $S^1$-action has no fixed points.) Let $\widetilde{X}$ be the universal covering of $X$ with the canonical $\pi_1(X)$-action. Then we have $b^{(2)}_{\varrho,n}(\widetilde{X}) = 0$.
\end{enumerate}
\end{proposition}
\begin{proof}
For a proof in the case $\varrho = 1$ we refer to \cite[Theorem 1.35 and Theorem 1.40]{Luc02}. The proof in the general case is very similar. We should add some remarks:
\begin{itemize}
\item The proofs of the assertions ``Poincar\'e duality'' and ``Zero-th $L^2$-Betti number'' use $L^2$-cohomology. $L^2$-cohomology is defined by the cellular $L^2$-cochain complex $C_{(2)}^{\varrho,*}(X) := \hom_{\Z G}(C^{cell}_*(X),l^2(G,\varrho))$ instead of the $L^2$-chain complex $C^{(2)}_{\varrho,*}(X)$. Using the Laplace operator we conclude that the Hilbert ${\mathcal N}(G,\varrho)$-modules $H^{(2)}_{\varrho,n}(C^{(2)}_{\varrho,*}(X))$ and $H_{(2)}^{\varrho,n}(C_{(2)}^{\varrho,*}(X))$ are isometrically isomorphic if we choose the same cellular $\Z G$-basis for $C^{cell}_*(X)$ (compare \cite[Lemma 1.18]{Luc02}).
\item For the ``Zero-th $L^2$-Betti number'' notice that we have $\varrho = 1$ if $G$ is finite.
\item For the statement about ``Restriction'' we can use the isomorphism of Hilbert ${\mathcal N}(H,\varrho|_H)$-chain complexes
\[
C^{(2)}_{\varrho|_H,*}(\res_H(X)) \to \res_H(C^{(2)}_{\varrho,*}(X))
\]
induced by the identity. It does not matter that this isomorphism is not an isometry in general because of Proposition \ref{prop_ex-seq} (\ref{prop_ex-seq-2}).
\end{itemize}
\end{proof}

In many examples we see that $b^{(2)}_{\varrho,n}(X)$ is independent of $\varrho$. This leads to the following conjecture.

\begin{conjecture} \label{conj-Betti}
Let $X$ be a finite free $G$-CW-complex.
For any group homomorphism $\varrho \colon G \to \R^{>0}$ we have
\[
b^{(2)}_{\varrho,n}(X) = b^{(2)}_n(X).
\]
\end{conjecture}

We denote by ${\mathcal C}$ the smallest class of groups which contains all free groups and is closed under directed unions and extensions with elementary amenable quotients.
This class of groups were introduced by Peter Linnell who proved the Atiyah Conjecture for all groups $G$ in this class ${\mathcal C}$ which satisfy $\lcm(G) < \infty$ (see \cite[Theorem 1.5]{Lin93}).

\begin{conjecture}[Atiyah Conjecture]
For a group $G$ we define $\lcm(G)$ as the least common multiple of the orders of the finite subgroups of $G$. If there is no bound on the orders of the finite subgroups of $G$ then we set $\lcm(G) := \infty$.\\
A group $G$ with $\lcm(G) < \infty$ satisfies the Atiyah Conjecture, if for any matrix $A \in M(m \times n; \C G)$ the von Neumann dimension of the kernel of the induced bounded $G$-operator
\[
r^{(2)}_A \colon l^2(G)^m \to l^2(G)^n, \, x \mapsto xA
\]
satisfies
\[
\lcm(G) \cdot \dim_{{\mathcal N}(G)} \big( \ker( r^{(2)}_A ) \big) \in \Z.
\]
\end{conjecture}

The strategy in Linnell's proof is to show that for any such group $G$ the division closure ${\mathcal D}(G)$ of the complex group ring $\C G$ in the algebra ${\mathcal U}(G)$ of operators affiliated to the group von Neumann algebra ${\mathcal N}(G)$ has the following two properties provided that $\lcm(G) < \infty$ holds:
\begin{itemize}
 \item The ring ${\mathcal D}(G)$ is semisimple.
 \item The composition
 \[
 \colim_{H \leq G \; \textrm{finite}} K_0(\C H) \to K_0(\C G) \to K_0({\mathcal D}(G))
 \]
 is surjective.
\end{itemize}
He shows that these two properties imply the Atiyah Conjecture. For details we refer to section ``10.2 A Strategy for the Proof of the Atiyah Conjecture'' in L\"{u}ck's book on $L^2$-invariants \cite{Luc02}.

\begin{proposition} \label{prop-Betti-independant}
Let $X$ be a finite free $G$-CW-complex and $\varrho \colon G \to \R^{>0}$ a group homomorphism.
Suppose that $G$ lies in the class ${\mathcal C}$ and has the property $\lcm(G) < \infty$.
Then
\[
b^{(2)}_{\varrho,n}(X) = b^{(2)}_n(X) \quad \textrm{for all} \; n.
\]
\end{proposition}
\begin{proof}
Because of the formula of the dimension function for short exact sequences (see property (\ref{dim_ses}) on page \pageref{dim_ses}) it suffices to show
\[
\dim_{{\mathcal N}(G,\varrho)}\big(\ker(r^{(2)}_A)\big) = \dim_{{\mathcal N}(G)}\big(\ker(r^{(2)}_A)\big)
\]
for any matrix $A \in M(m \times n; \C G)$.
Notice that
\[
\dim_{{\mathcal N}(G,\varrho)}\big(\ker(r^{(2)}_A)\big) = \dim_{{\mathcal N}(G)}\big(\ker(r^{(2)}_{\Phi_\varrho(A)})\big).
\]
Because of the induction properties we can assume that $G$ is generated by the support of the entries of the matrix $A$. This means that $G$ is finitely generated.
Since the class ${\mathcal C}$ is closed under taking subgroups, the kernel $\ker(\varrho) < G$ lies in ${\mathcal C}$. In particular, the ring ${\mathcal D}(\ker(\varrho))$ is semisimple. Consider the crossed product ring ${\mathcal D}(\ker(\varrho)) * G/\ker(\varrho) \subset {\mathcal U}(G)$ which is the smallest subring containing ${\mathcal D}(\ker(\varrho))$ and $\C G$. The isomorphism $\Phi_\varrho \colon \C G \to \C G$ extends to an isomorphism $\Phi_\varrho \colon {\mathcal D}(\ker(\varrho)) * G/\ker(\varrho) \to {\mathcal D}(\ker(\varrho)) * G/\ker(\varrho)$ which is the identity on ${\mathcal D}(\ker(\varrho))$. \cite[Lemma 10.69]{Luc02} tells us that the ring ${\mathcal D}(\ker(\varrho)) * G/\ker(\varrho)$ satisfies the Ore condition with respect to the multiplicative set of all non-zero divisors $\NZD({\mathcal D}(\ker(\varrho)) * G/\ker(\varrho))$. Moreover, the Ore localization agrees with ${\mathcal D}(G)$. We conclude that the isomorphism $\Phi_\varrho \colon {\mathcal D}(\ker(\varrho)) * G/\ker(\varrho) \to {\mathcal D}(\ker(\varrho)) * G/\ker(\varrho)$ extends to an isomorphism $\Phi_\varrho \colon {\mathcal D}(G) \to {\mathcal D}(G)$.
There is a dimension function $\dim_{{\mathcal U}(G)} \colon K_0({\mathcal U}(G)) \to \R$ (see \cite[section 8.3]{Luc02}). We conclude from \cite[Theorem 8.29 and Theorem 6.24 (4)]{Luc02} that
\[
\dim_{{\mathcal N}(G)}\big(\ker(r^{(2)}_A)\big) = m - n + \dim_{{\mathcal U}(G)}\big({\mathcal U}(G) \otimes_{\C G} \coker(r_A \colon {\C G}^m \to{\C G}^n )\big).
\]
Hence it suffices to show
\[
\dim_{{\mathcal U}(G)}\big({\mathcal U}(G) \otimes_{\C G} \coker(r_A)\big) = \dim_{{\mathcal U}(G)}\big({\mathcal U}(G) \otimes_{\C G} \coker(r_{\Phi_\varrho(A)})\big).
\]
Since ${\mathcal D}(G)$ is semisimple, every finitely generated ${\mathcal D}(G)$-module is projective.
Consider the group homomorphism
\[
\alpha \colon K_0({\mathcal D}(G)) \longrightarrow K_0({\mathcal U}(G)) \stackrel{\dim_{{\mathcal U}(G)}}{\longrightarrow} \R.
\]
Let $i \colon H \hookrightarrow G$ be the inclusion of a finite subgroup and let $P$ be a finitely generated projective $\C H$-module. Choose a matrix $F \in M(n \times n; \C H)$ with $F^2 = F$ and $\im(r_F) = P$. We calculate
\begin{align*}
\alpha \circ {\Phi_\varrho}_* \circ i_* (P) & = \tr_{\C H}(\Phi_\varrho(F)) = \sum_{i=1}^n \langle \Phi_\varrho(F_{ii}) , e \rangle = \\
& = \sum_{i=1}^n \langle F_{ii} , e \rangle = \tr_{\C H}(F) = \alpha \circ i_* (P).
\end{align*}
This shows $\alpha \circ {\Phi_\varrho}_* \circ i_* = \alpha \circ i_*$.
The surjectivity of the composition
\[
\colim_{H \leq G \; \textrm{finite}} K_0(\C H) \to K_0(\C G) \to K_0({\mathcal D}(G))
\]
yields $\alpha \circ {\Phi_\varrho}_* = \alpha$.
We finally obtain
\begin{align*}
& \dim_{{\mathcal U}(G)}\big({\mathcal U}(G) \otimes_{\C G} \coker(r_A)\big) = \alpha\big([{\mathcal D}(G) \otimes_{\C G} \coker(r_A)]\big) = \\
& = \alpha \circ {\Phi_\varrho}_* \big([{\mathcal D}(G) \otimes_{\C G} \coker(r_A)]\big) =
\alpha\big([{\mathcal D}(G) \otimes_{\C G} \coker(r_{\Phi_\varrho(A)})]\big) = \\
& = \dim_{{\mathcal U}(G)}\big({\mathcal U}(G) \otimes_{\C G} \coker(r_{\Phi_\varrho(A)})\big).
\end{align*}
\end{proof}

\subsection{Weighted Novikov--Shubin invariants}

\begin{definition} \label{def-spec_dens}
Let $f \colon U \to V$ be a map of finitely generated Hilbert ${\mathcal N}(G,\varrho)$-modules.
Denote by $\{E_\lambda^{f^* f} \colon U \to U \, | \, \lambda \in \R\}$ the family of spectral projections of the positive endomorphism $f^*f \colon U \to U$.
We define the spectral density function of $f$ by
\[
F(f)(\lambda) = \dim_{{\mathcal N}(G,\varrho)} \big( \im(E_{\lambda^2}^{f^* f}) \big).
\]
and the Novikov--Shubin invariant of $f$ by
\[
\alpha_{\varrho}(f) := \liminf_{\lambda \to 0+} \frac{\ln\big( F_n(f)(\lambda) - F_n(f)(0) \big)}{\ln(\lambda)} \in [0,\infty],
\]
provided that $F_n(f)(\lambda) > F_n(f)(0)$ for all $\lambda > 0$. Otherwise, we set $\alpha_{\varrho}(f) := \infty^+$.
(Here, $\infty^+$ is a new formal symbol.)
\end{definition}

Notice that $\alpha_{\varrho}(f) = \infty^+$ if $f$ is an isomorphism.


\begin{definition}
Let $(C_*,c_*)$ be a finite Hilbert ${\mathcal N}(G,\varrho)$-chain complex.
We define its Novikov--Shubin invariants by
\[
\alpha_{\varrho,n}(C_*) := \alpha_{\varrho}(c_n) \in [0,\infty] \cup \{\infty^+\}.
\]
\end{definition}

Notice that we have $\alpha_{\varrho}(f) = \alpha(\Psi_\varrho(f))$ and $\alpha_{\varrho,n}(C_*) = \alpha_n(\varrho(C_*))$.

\begin{proposition} \label{prop_NS-homotopy}
Let $f_* \colon C_* \to D_*$ be a homotopy equivalence of finite Hilbert ${\mathcal N}(G,\varrho)$-chain complexes. Then we have
\[
\alpha_{\varrho,n}(C_*) = \alpha_{\varrho,n}(D_*).
\]
\end{proposition}
\begin{proof}
Using \cite[Theorem 2.19]{Luc02} we obtain
\[
\alpha_{\varrho,n}(C_*) = \alpha_n(\Psi_\varrho(C_*)) = \alpha_n(\Psi_\varrho(D_*)) = \alpha_{\varrho,n}(D_*).
\]
\end{proof}

\begin{definition} \label{def-NS-Inv}
Let $X$ be a finite free $G$-CW-complex and $\varrho \colon G \to \R^{>0}$ a group homomorphism.
We define the weighted Novikov--Shubin invariants of $X$ by
\[
\alpha_{\varrho,n}(X) := \alpha_{\varrho,n}\big(C^{(2)}_{\varrho,*}(X)\big)
\]
where $C^{(2)}_{\varrho,*}(X) := l^2(G,\varrho) \otimes_{\Z G} C^{cell}_*(X)$ is the Hilbert ${\mathcal N}(G,\varrho)$-chain complex defined in Definition \ref{def-L2-Betti}.
\end{definition}
The definition of $C^{(2)}_{\varrho,*}(X)$ depends on the choice of a cellular basis for $C^{cell}_*(X)$. Using Proposition \ref{prop_NS-homotopy} we conclude as in the case of weighted $L^2$-Betti numbers that $\alpha_{\varrho,n}(X)$ is independent of the choice of the cellular basis.

\begin{proposition} \label{properties-NS}
\begin{enumerate}
\item Homotopy invariance \\
Let $f \colon X \to Y$  be a $G$-map of finite free $G$-CW-complexes and $\varrho \colon G \to \R^{>0}$ a group homomorphism. If the map induced on homology with complex coefficients $H_n(f;\C) \colon H_n(X;\C) \to H_n(Y;\C)$ is bijective for $n < d$ then
\[
\alpha_{\varrho,n}(X) = \alpha_{\varrho,n}(Y) \quad \textrm{for} \; n \leq d.
\]
In particular, if $f$ is a weak homotopy equivalence then
\[
\alpha_{\varrho,n}(X) = \alpha_{\varrho,n}(Y) \quad \textrm{for all} \; n.
\]
\item Poincar\'e duality \\
Let $M$ be a cocompact free proper $G$-manifold of dimension $m$ which is orientable. Then
\[
\alpha_{\varrho,n}(M) = \alpha_{\varrho,m+1-n}(M, \partial M).
\]
\item Connected sum \\
Let $M_1, M_2$ be compact connected $m$-dimensional manifolds with $m \geq 3$ and $M = M_1 \# M_2$ be their connected sum.
Let $\varrho \colon \pi_1(M) \cong \pi_1(M_1) \ast \pi_1(M_2) \to \R^{>0}$ be a group homomorphism.
Then
\[
\alpha_{\varrho,n}(\widetilde{M}) = \min\{ \alpha_{\varrho|_{\pi_1(M_1),n}}(\widetilde{M_1}), \alpha_{\varrho|_{\pi_1(M_2),n}}(\widetilde{M_2}) \} \quad \textrm{for} \; n \geq 2.
\]
\item First Novikov--Shubin invariant \\
Let $X$ be a connected finite free $G$-CW-complex.
Then $G$ is finitely generated and following statements hold.
\begin{enumerate}
\item $\alpha_{\varrho,1}(X)$ is finite if and only if $\varrho = 1$ and $G$ is infinite virtually nilpotent. In this case $\alpha_{\varrho,1}(X) = \alpha_1(X)$ is the growth rate of $G$.
\item $\alpha_{\varrho,1}(X) = \infty$ if and only if $\varrho = 1$ and $G$ is finite or non-amenable.
\item $\alpha_{\varrho,1}(X) = \infty^+$ if and only if $\varrho \neq 1$ or $G$ is amenable and not virtually nilpotent.
\end{enumerate}
\item Restriction \\
Let $X$ be a finite free $G$-CW-complex and let $H < G$ be a subgroup of finite index.
Let $\res_H(X)$ be the finite free $H$-CW-complex obtained from $X$ by restricting the $G$-action to an $H$-action.
We have
\[
\alpha_{\varrho|_H,n}(\res_H(X)) = \alpha_{\varrho,n}(X).
\]
\item Induction \\
Let $X$ be a finite free $H$-CW-complex and let $H < G$ be a subgroup.
Then $G \times_H X$ is a finite free $G$-CW-complex and
\[
\alpha_{\varrho,n}(G \times_H X) = \alpha_{\varrho|_H,n}(X).
\]
\item $S^1$-actions \\
Let $X$ be a connected $S^1$-CW-complex of finite type, for instance a connected compact manifold with smooth $S^1$-action. Suppose that for one orbit $S^1/H$ (and hence for all orbits) the inclusion into $X$ induces a map on $\pi_1$ with infinite image. (In particular the $S^1$-action has no fixed points.) Let $\widetilde X$ be the universal covering of $X$ with the canonical $\pi_1(X)$-action. Then we have $\alpha_{\varrho,n}(\widetilde{X}) \geq 1$.
\end{enumerate}
\end{proposition}
\begin{proof}
Except for the assertion ``First Novikov--Shubin invariant'' the proof is essentially the same as in the case $\varrho = 1$. For this case we refer to \cite[Theorem 2.55 and Theorem 2.61]{Luc02}.\\
To the assertion ``First Novikov--Shubin invariant'': We know from \cite[Theorem 2.55 (5)]{Luc02} that $G$ is finitely generated and that the assertion holds in the case $\varrho = 1$. Now consider the case $\varrho \neq 1$. We have to show $\alpha_{\varrho,1}(X) = \infty^+$. Following the proof of \cite[Lemma 2.45]{Luc02} we conclude $\alpha_{\varrho,1}(X) = \alpha_{\varrho}(c_S)$. Here $S$ is any finite set of generators of $G$ and $c_S$ is defined by
\[
c_S \colon \xymatrix@1@+3pc{\bigoplus_{s \in S} l^2(G,\varrho) \ar[r]^-{\bigoplus_{s \in S} r^{(2)}_{s-1}} & l^2(G,\varrho)}
\]
where $r^{(2)}_{s-1}$ is right multiplication with $(s-1)$. We will show that $c_S|_{\ker(c_S)^\perp} \colon \ker(c_S) \to l^2(G,\varrho)$ is invertible. This would imply
\[
\alpha_{\varrho,1}(X) = \alpha_{\varrho}(c_S) = \alpha_{\varrho}(c_S|_{\ker(c_S)^\perp}) = \infty^+.
\]
Since $\varrho \neq 1$, there exist $s' \in S$ with $\varrho(s') \neq 1$. Notice that $r^{(2)}_{s'-1} \colon l^2(G,\varrho) \to l^2(G,\varrho)$ is invertible. If $\varrho(s') < 1$ then the inverse is given by right multiplication with $\sum_{k=0}^\infty {s'}^k$. This is a bounded operator because
\[
\big\| r^{(2)}_{\sum_{k=0}^\infty {s'}^k} \big\| \leq \sum_{k=0}^\infty \| r^{(2)}_{s'} \|^k = \sum_{k=0}^\infty \sqrt{\varrho(s')}^k = \frac{1}{1-\sqrt{\varrho(s')}}.
\]
If $\varrho(s') > 1$ then ${r^{(2)}_{s'-1}}^{-1} = r^{(2)}_{-{s'}^{-1}} \circ {r^{(2)}_{{s'}^{-1}-1}}^{-1}$.
We define the map
\[
\alpha \colon \xymatrix@1@+1pc{l^2(G,\varrho) \ar[r]^-{{r^{(2)}_{s'-1}}^{-1}}& l^2(G,\varrho) \ar[r]^-{i_{s'}} & \bigoplus_{s \in S} l^2(G,\varrho) \ar[r]^-{\pr}& \ker(c_S)^\perp}.
\]
We have $c_S|_{\ker(c_S)^\perp} \circ \alpha = \id$. Since $c_S|_{\ker(c_S)^\perp}$ is injective, we conclude $\alpha \circ c_S|_{\ker(c_S)^\perp} = \id$. This shows that $c_S|_{\ker(c_S)^\perp}$ is invertible. Hence $\alpha_{\varrho,1}(X) = \alpha_{\varrho}(c_S|_{\ker(c_S)^\perp}) = \infty^+$.
\end{proof}

\subsection{Weighted $L^2$-torsion}

Let $f \colon U \to V$ be a map of finitely generated Hilbert ${\mathcal N}(G,\varrho)$-modules.
Its spectral density function $F(f)(\lambda)$ (see Definition \ref{def-spec_dens}) is monotonous and right-continuous. It defines a measure on the Borel $\sigma$-algebra on $\R$ which is uniquely determined by $dF(f)((a,b]) := F(f)(b)-F(f)(a)$ for $a<b$.

\begin{definition}
Let $f \colon U \to V$ be a map of finitely generated Hilbert ${\mathcal N}(G,\varrho)$-modules.
We define the Fuglede--Kadison determinant of $f$ by
\[
{\det}_{{\mathcal N}(G,\varrho)}(f) := \exp \big( \int_{0^+}^\infty \ln(\lambda) \, dF(f)(\lambda) \big)
\]
if $\int_{0^+}^\infty \ln(\lambda) \, dF(f)(\lambda) > -\infty$ and by $\det_{{\mathcal N}(G,\varrho)}(f) := 0$ otherwise.
\end{definition}

If $G$ is the trivial group and $f \colon \C^n \to \C^n$ is an isomorphism then we have $\det_{{\mathcal N}(G,\varrho)}(f) = |\det_\C(f)|$.

The following proposition which is the analogue to \cite[Theorem 3.14]{Luc02} states the main properties of the Fuglede--Kadison determinant.
\begin{proposition} \label{properties-det}
Let $\varrho \colon G \to \R^{>0}$ be a group homomorphism.
\begin{enumerate}
\item Let $f \colon U \to V$ and $g \colon V \to W$ be morphisms of finitely generated Hilbert ${\mathcal N}(G,\varrho)$-modules such that $f$ has dense image and $g$ is injective. Then
\[
{\det}_{{\mathcal N}(G,\varrho)}(g \circ f) = {\det}_{{\mathcal N}(G,\varrho)}(f) \cdot {\det}_{{\mathcal N}(G,\varrho)}(g).
\] \label{properties-det-1}
\item Let $f_1 \colon U_1 \to V_1$, $f_2 \colon U_2 \to V_2$ and $f_3 \colon U_2 \to V_1$ be morphisms of finitely generated Hilbert ${\mathcal N}(G,\varrho)$-modules such that $f_1$ has dense image and $f_2$
is injective. Then
\[
{\det}_{{\mathcal N}(G,\varrho)}(\begin{pmatrix} f_1 & f_3 \cr 0 & f_2 \end{pmatrix}) = {\det}_{{\mathcal N}(G,\varrho)}(f_1) \cdot {\det}_{{\mathcal N}(G,\varrho)}(f_2).
\] \label{properties-det-2}
\item Let $f \colon U \to V$ be a morphism of finitely generated Hilbert ${\mathcal N}(G,\varrho)$-modules. Then
\[
{\det}_{{\mathcal N}(G,\varrho)}(f) = {\det}_{{\mathcal N}(G,\varrho)}(f^*).
\] \label{properties-det-3}
\item If the Novikov--Shubin invariant of the morphism $f \colon U \to V$ of finitely generated Hilbert ${\mathcal N}(G,\varrho)$-modules satisfies $\alpha_\varrho(f) > 0$ then
\[
{\det}_{{\mathcal N}(G,\varrho)}(f) > 0.
\] \label{properties-det-4}
\item Let $f \colon U \to V$ be a morphism of finitely generated Hilbert ${\mathcal N}(G,\varrho)$-modules and let $H < G$ be a subgroup of finite index. Then
\[
{\det}_{{\mathcal N}(H,\varrho|_H)}(\res_H(f)) = {\det}_{{\mathcal N}(G,\varrho)}(f)^{[G:H]}.
\] \label{properties-det-5}
\item Let $\varrho \colon G \to \R^{>0}$ be a group homomorphism and $H < G$ a subgroup. Let $f \colon U \to V$ be a morphism of finitely generated Hilbert ${\mathcal N}(H,\varrho|_H)$-modules. Then
\[
{\det}_{{\mathcal N}(G,\varrho)}(\ind_G(f)) = {\det}_{{\mathcal N}(H,\varrho|_H)}(f).
\] \label{properties-det-6}
\end{enumerate}
\end{proposition}
\begin{proof}
The first five statements are direct consequences of \cite[Theorem 3.14]{Luc02} because of the equation $\det_{{\mathcal N}(G,\varrho)}(f) = \det_{{\mathcal N}(G)}(\Psi_\varrho(f))$.
The statement (\ref{properties-det-6}) follows from the equation $\ind_G(E_{\lambda^2}^{f^* f}) = E_{\lambda^2}^{\ind_G(f)^* \ind_G(f)}$.
\end{proof}

\begin{definition}
Let $(C_*,c_*)$ be a finite Hilbert ${\mathcal N}(G,\varrho)$-chain complex.
Suppose that $b^{(2)}_{\varrho,n}(C_*) = 0$ and $\det_{{\mathcal N}(G,\varrho)}(c_n) > 0$ for all $n$.
We define its $L^2$-torsion by
\[
\rho^{(2)}_\varrho(C_*) := - \sum_{n \in \Z} (-1)^n \cdot \ln {\det}_{{\mathcal N}(G,\varrho)}(c_n) \in \R.
\]
\end{definition}

Notice that $\det_{{\mathcal N}(G,\varrho)}(f) = \det_{{\mathcal N}(G)}(\Psi_\varrho(f))$ and $\rho^{(2)}_\varrho(C_*) = \rho^{(2)}(\Psi_\varrho(C_*))$.

The following proposition is a direct consequence of \cite[Theorem 3.35 (1),(5), Lemma 3.41 and Lemma 3.44]{Luc02}.
\begin{proposition} \label{prop_torsion}
\begin{enumerate}
\item Let $0 \to C_* \stackrel{i_*}{\longrightarrow} D_* \stackrel{p_*}{\longrightarrow} E_* \to 0$ be an exact sequence of finite Hilbert ${\mathcal N}(G,\varrho)$-chain complexes.
Suppose that two of the Hilbert ${\mathcal N}(G,\varrho)$-chain complexes have the properties that the $L^2$-Betti numbers vanish and that the determinant of the differentials are positive.
Then all three have these properties, $\det_{{\mathcal N}(G,\varrho)}(i_n), \det_{{\mathcal N}(G,\varrho)}(p_n) > 0$ and
\begin{align*}
& \rho^{(2)}_\varrho(C_*) - \rho^{(2)}_\varrho(D_*) + \rho^{(2)}_\varrho(E_*) = \\
& \sum_{n \in \Z} (-1)^n \cdot \ln {\det}_{{\mathcal N}(G,\varrho)}(p_n) - \sum_{n \in \Z} (-1)^n \cdot \ln {\det}_{{\mathcal N}(G,\varrho)}(i_n).
\end{align*} \label{prop_torsion-1}
\item Let $C_*$ and $D_*$ be finite Hilbert ${\mathcal N}(G,\varrho)$-chain complexes of determinant class and $f_* \colon C_* \to D_*$ be a homotopy equivalence.
Suppose that one of these Hilbert ${\mathcal N}(G,\varrho)$-chain complexes has the properties that the $L^2$-Betti numbers vanish and that the determinant of the differentials are positive.
Then both $C_*$, $D_*$ and the mapping cone $\cone_*(f_*)$ have these properties and
\[
\rho^{(2)}_\varrho(C_*) - \rho^{(2)}_\varrho(D_*) = - \rho^{(2)}_\varrho(\cone_*(f_*)).
\]
Moreover, we have $\det_{{\mathcal N}(G,\varrho)}(f_n) > 0$. If $f_*$ is a chain isomorphism then
\[
\rho^{(2)}_\varrho(C_*) - \rho^{(2)}_\varrho(D_*) = \sum_{n \in \Z} (-1)^n \cdot \ln {\det}_{{\mathcal N}(G,\varrho)}(f_n).
\] \label{prop_torsion-2}
\item Let $C_*$ be a finite Hilbert ${\mathcal N}(G,\varrho)$-chain complex with a chain contraction $\gamma_*$.
Suppose that $\det_{{\mathcal N}(G,\varrho)}(c_n) > 0$ for all $n$.
Then the map
\[
(c_* + \gamma_*)_{odd} \colon \oplus_{n \in \Z} C_{2n+1} \to \oplus_{n \in \Z} C_{2n}
\]
is an isomorphisms with $\det_{{\mathcal N}(G,\varrho)}((c_* + \gamma_*)_{odd}) > 0$ and we get
\[
\rho^{(2)}_\varrho(C_*) = \ln {\det}_{{\mathcal N}(G,\varrho)}((c_* + \gamma_*)_{odd}).
\] \label{prop_torsion-3}
\end{enumerate}
\end{proposition}

Let $X$ be a finite free $G$-CW-complex. We would like to define the weighted $L^2$-torsion of $X$ by
\[
\rho^{(2)}_\varrho(X) := \rho^{(2)}_\varrho\big(C^{(2)}_{\varrho,*}(X)\big)
\]
where $C^{(2)}_{\varrho,*}(X) := l^2(G,\varrho) \otimes_{\Z G} C^{cell}_*(X)$ is the Hilbert ${\mathcal N}(G,\varrho)$-chain complex introduced in Definition \ref{def-L2-Betti}.
But it turns out that this is not well-defined. The problem is that $C^{(2)}_{\varrho,*}(X)$ depends on the choice of a cellular basis for $C^{cell}_*(X)$. The cellular basis is only unique up to permutation and multiplication with $\pm g \in \Z G$ for $g \in G$.
Nevertheless, we have the following lemma.

\begin{lemma}
For different choices of a cellular basis the values $\rho^{(2)}_\varrho\big(C^{(2)}_{\varrho,*}(X)\big)$ differ by $\frac{1}{2} \cdot \ln(\varrho(g))$ for some $g \in G$ which does not depend on $\varrho$.
\end{lemma}
\begin{proof}
Let $C_*, D_*$ be two ${\mathcal N}(G,\varrho)$-chain complexes of the shape $l^2(G,\varrho) \otimes_{\Z G} C^{cell}_*(X)$.
Since the cellular basis is unique up to permutation and multiplication with $\pm g \in \Z G$ for $g \in G$, we have a chain map $f_* \colon C_* \to D_*$ with the properties that $f_n$ is invertible and
\[
{\det}_{{\mathcal N}(G,\varrho)}(f_n) = {\det}_{{\mathcal N}(G,\varrho)}(g_n) = \sqrt{\varrho(g_n)}.
\]
for some elements $g_n \in G$ (compare Lemma \ref{properties-det}).
From \ref{prop_torsion} (\ref{prop_torsion-2}) we conclude
\[
\rho^{(2)}_\varrho(C_*) - \rho^{(2)}_\varrho(D_*) = \sum_{n \in \Z} (-1)^n \cdot \ln {\det}_{{\mathcal N}(G,\varrho)}(f_n) = \frac{1}{2} \cdot \ln(\varrho(g))
\]
with $g := g_0 \cdot {g_1}^{-1} \cdot g_2 \cdot {g_3}^{-1} \cdots$.
\end{proof}

\begin{definition}
Let $X$ be a finite free $G$-CW-complex and $\varrho \colon G \to \R^{>0}$ a group homomorphism.
For $x \in \R$ we set
\[
\varrho^x \colon G \to \R^{>0}, g \mapsto \varrho(g)^x.
\]
Notice that $\{ x \mapsto x \cdot \ln(\varrho(g)) \mid g \in G \}$ is a subgroup of the additive group $\map(\R,\R)$.
Suppose that $X$ is det-$L^2$-acyclic with respect to $\varrho$ i.e. $b^{(2)}_{\varrho^x,n}(X) = 0$ and $\det_{{\mathcal N}(G,\varrho^x)}(c^{(2)}_n(X)) > 0$ for all $n \in \Z$ and $x \in \R$.
We define the weighted $L^2$-torsion of $X$
\[
\rho^{(2)}_\varrho(X) \in \map(\R,\R) / \{ x \mapsto \frac{x}{2} \cdot \ln(\varrho(g)) \mid g \in G \}
\]
by $\rho^{(2)}_\varrho(X)(x) := [\rho^{(2)}_{\varrho^x}(C^{(2)}_{\varrho^x,*}(X))]$.
\end{definition}
Notice that $\rho^{(2)}_\varrho(X)(0) = \rho^{(2)}(X)$. We have $\rho^{(2)}_1(X)(x) = \rho^{(2)}(X)$ for all $x \in \R$.

\begin{example}
As in Example \ref{ex-S1-Betti} we consider the universal covering of $S^1$.
Let $\varrho \colon \pi_1(S^1) \to \R^{>0}$ be a group homomorphism.
The Hilbert ${\mathcal N}(G,\varrho)$-chain complex $C^{(2)}_{\varrho,*}(\widetilde{S^1})$ has the shape
\[
l^2(G,\varrho) \stackrel{\cdot (t-1)}{\longrightarrow} l^2(G,\varrho)
\]
with $\langle t \rangle = \pi_1(S^1) \cong \Z$.
We calculate
\begin{align*}
& \rho^{(2)}_{\varrho}(\widetilde{S^1})(x) = \ln {\det}_{{\mathcal N}(G,\varrho^x)}(r^{(2)}_{t-1}) = \ln {\det}_{{\mathcal N}(G)}(\Psi_{\varrho^x}(r^{(2)}_{t-1})) = \ln {\det}_{{\mathcal N}(G)}(r^{(2)}_{\Phi_{\varrho^x}(t-1)}) = \\
& = \ln {\det}_{{\mathcal N}(G)}(r^{(2)}_{\sqrt{\varrho^x(t)} \cdot t-1}) = \ln \max\{ \sqrt{\varrho^x(t)} ,  1\} = \max\{\frac{x}{2} \cdot \ln(\varrho(t)) , 0\}.
\end{align*}
(compare \cite[Example 3.22]{Luc02}).
\end{example}

We have a group homomorphism
\begin{eqnarray*}
A_\varrho \colon \Wh(G) & \to & \map(\R,\R) / \{ x \mapsto \frac{x}{2} \cdot \ln(\varrho(g)) \mid g \in G \} \\
{[M]} & \mapsto & x \mapsto \ln {\det}_{{\mathcal N}(G,\varrho^x)}(r^{(2)}_M)
\end{eqnarray*}
This is well-defined because of Proposition \ref{properties-det} (\ref{properties-det-1}), \ref{properties-det-2} and the equation
\[
\ln {\det}_{{\mathcal N}(G,\varrho^x)}(r^{(2)}_{\pm g}) = \frac{x}{2} \cdot \ln(\varrho(g)).
\]

The following proposition states the basic properties of the weighted $L^2$-torsion of finite free $G$-CW-complexes.
\begin{proposition} \label{properties-torsion}
\begin{enumerate}
\item Homotopy invariance \\ \label{properties-torsion-1}
Let $f \colon X \to Y$ be a $G$-homotopy equivalence of finite free $G$-CW-complexes.
Let $\varrho \colon \pi_1(X) \cong \pi_1(Y) \to \R^{>0}$ be a group homomorphism.
Suppose that $X$ or $Y$ is det-$L^2$-acyclic w.r.t. $\varrho$. Then both $X$ and $Y$ are det-$L^2$-acyclic w.r.t. $\varrho$ and
\[
 \rho^{(2)}_\varrho(Y) - \rho^{(2)}_\varrho(X) = A_\varrho(\tau(f))
\]
where $\tau(f) \in \Wh(G)$ denotes the Whitehead torsion of $f$.
\item Sum formula \\ \label{properties-torsion-2}
Consider the $G$-pushout of finite free $G$-CW-complexes such that $j_1$ is an inclusion of $G$-CW-complexes, $j_2$ is cellular and $X$ inherits its $G$-CW-complex structure from $X_0$, $X_1$ and $X_2$.
\begin{equation*}
\xymatrix{
X_0 \ar[r]^{j_1} \ar[d]_{j_2} & X_1 \ar[d]_{i_1} \\
X_2 \ar[r]^{i_2} & X }
\end{equation*}
Assume that three of the $G$-CW-complexes $X_0$, $X_1$, $X_2$ and $X$ are det-$L^2$-acyclic w.r.t. $\varrho \colon G \to \R^{>0}$.
Then all four $G$-CW-complexes $X_0$, $X_1$, $X_2$ and $X$ are det-$L^2$-acyclic w.r.t. $\varrho$ and
\[
\rho^{(2)}_\varrho(X) = \rho^{(2)}_\varrho(X_1) + \rho^{(2)}_\varrho(X_2) - \rho^{(2)}_\varrho(X_0).
\]
\item Poincar\'e duality \\ \label{properties-torsion-3}
Let $M$ be a cocompact free proper $G$-manifold without boundary of even dimension which is orientable.
Suppose that $G$ acts orientation preserving on $M$ and that $M$ is det-$L^2$-acyclic w.r.t. $\varrho \colon G \to \R^{>0}$.
Then
\[
\rho^{(2)}_\varrho(M) = 0.
\]
\item Restriction \\ \label{properties-torsion-4}
Let $X$ be a finite free $G$-CW-complex and let $H < G$ be a subgroup of finite index.
Let $\res_H(X)$ be the finite $H$-CW-complex obtained from $X$ by restricting the $G$-action to an $H$-action.
Then $X$ is det-$L^2$-acyclic w.r.t. $\varrho \colon G \to \R^{>0}$ if and only if $\res_H(X)$ is det-$L^2$-acyclic w.r.t. $\varrho|_H$, and in this case
\[
\rho^{(2)}_{\varrho|_H}(\res_H(X)) = [G:H] \cdot \rho^{(2)}_\varrho(X).
\]
\item Induction \\ \label{properties-torsion-5}
Let $H$ be a subgroup of $G$ and let $X$ be a finite free $H$-CW-complex.
Then the finite free $G$-CW-complex $G \times_H X$ is det-$L^2$-acyclic w.r.t. $\varrho \colon G \to \R^{>0}$ if and only if $X$ is det-$L^2$-acyclic w.r.t. $\varrho|_H$, and in this case
\[
\rho^{(2)}_\varrho(G \times_H X) = \rho^{(2)}_{\varrho|_H}(X).
\]
\item Positive Novikov--Shubin invariants and determinant class \\ \label{properties-torsion-6}
If $X$ is a finite free $G$-CW-complex with $b^{(2)}_{\varrho,n}(X) = 0$ and $\alpha^{(2)}_{\varrho,n}(X) > 0$ for all $n$, then $X$ is det-$L^2$-acyclic w.r.t. $\varrho$.
\item $S^1$-actions \\ \label{properties-torsion-7}
Let $X$ be a connected $S^1$-CW-complex of finite type and $\varrho \colon \pi_1(X) \to \R^{>0}$ a group homomorphism.
Suppose that for one orbit $S^1/H$ (and hence for all orbits) the inclusion into $X$ induces a map on $\pi_1$ with infinite image.
(In particular the $S^1$-action has no fixed points.)
Let $\widetilde X$ be the universal covering of $X$ with the canonical $\pi_1(X)$-action.
Then $\widetilde X$ is det-$L^2$-acyclic w.r.t. $\varrho$ and there exists $g \in G$ with
\[
\rho^{(2)}_\varrho(\widetilde{X})(x) = \max\{\frac{x}{2} \cdot \ln(\varrho(g)) , 0\}.
\]
\end{enumerate}
\end{proposition}
\begin{proof}
The proof is essentially the same as the proofs of \cite[Theorem 3.93 and Theorem 3.105]{Luc02}.
\begin{enumerate}
\item This follows from Proposition \ref{prop_torsion} (\ref{prop_torsion-2}) and (\ref{prop_torsion-3}).
\item This follows from Proposition \ref{prop_torsion} (\ref{prop_torsion-1}) applied to the exact sequence of Hilbert ${\mathcal N}(G,\varrho)$-chain complexes
\[
0 \to C^{(2)}_{\varrho,*}(X_0) \longrightarrow C^{(2)}_{\varrho,*}(X_1) \oplus C^{(2)}_{\varrho,*}(X_2) \longrightarrow C^{(2)}_{\varrho,*}(X) \to 0.
\]
\item The proof is analogous to the proof of \cite[Theorem 3.93 (3)]{Luc02}.
\item Fix a cellular $\Z G$-basis for $C^{cell}_*(X)$ and a right transversal $T$ of $H$ in $G$. Using the $\Z H$-isomorphism $\oplus_{g \in T} \Z H \to \Z G, (x_g) \mapsto \sum_{g \in T} x_g \cdot g$ we obtain a cellular $\Z H$-basis for $C^{cell}_*(\res_H(X))$.
The isomorphism
\[
\alpha \colon \oplus_{g \in T} l^2(H,\varrho|_H) \to l^2(G,\varrho), (v_g) \mapsto \sum_{g \in T} v_g \cdot g
\]
induces a chain isomorphism $f_* \colon C^{(2)}_{\varrho|_H,*}(\res_H(X)) \to \res_H(C^{(2)}_{\varrho,*}(X))$ with $f_n = \diag(\alpha)$.
Proposition \ref{prop_torsion} (\ref{prop_torsion-2}) tells us that
\[
\rho^{(2)}_{\varrho|_H}(\res_H(X)) - \rho^{(2)}_{\varrho|_H}\big(\res_H(C^{(2)}_{\varrho,*}(X))\big) = \sum_{n \in \Z} (-1)^n \cdot \ln {\det}_{{\mathcal N}(H,\varrho|_H)}(f_n).
\]
Proposition \ref{properties-det} (\ref{properties-det-5}) implies
\[
\rho^{(2)}_{\varrho|_H}\big(\res_H(C^{(2)}_{\varrho,*}(X))\big) = [G:H] \cdot \rho^{(2)}_\varrho(X).
\]
We obtain
\begin{align*}
& \rho^{(2)}_{\varrho|_H}(\res_H(X)) - [G:H] \cdot \rho^{(2)}_\varrho(X) = \sum_{n \in \Z} (-1)^n \cdot \ln {\det}_{{\mathcal N}(H,\varrho|_H)}(f_n) = \\
& \sum_{n \in \Z} (-1)^n \cdot \ln {\det}_{{\mathcal N}(H,\varrho|_H)}(\alpha)^{\beta_n(G \setminus X)} = \chi(G \setminus X) \cdot \ln {\det}_{{\mathcal N}(H,\varrho|_H)}(\alpha) = 0
\end{align*}
where $\beta_n(G \setminus X)$ is the number of $n$-cells of $G \setminus X$ and $\chi(G \setminus X)$ is the Euler characteristic of the finite CW-complex $G \setminus X$.
Notice that $\chi(G \setminus X) = \sum_{n \geq 0} (-1)^n \cdot b^{(2)}_{\varrho,n}(X) = 0$ (see Proposition \ref{properties-Betti} (\ref{properties-Betti-2})).
\item Fix a cellular $\Z H$-basis for $C^{cell}_*(X)$. We obtain an induced cellular $\Z G$-basis for $C^{cell}_*(G \times_H X) = \Z G \otimes_{\Z H} C^{cell}_*(X)$.
Now the statement follows from the equation $C^{(2)}_{\varrho,*}(G \times_H X) = \ind_G(C^{(2)}_{\varrho|_H,*}(X))$ and Proposition \ref{properties-det} (\ref{properties-det-6}).
\item This follows from Proposition \ref{properties-det} (\ref{properties-det-4}).
\item We consider
\[
[\rho^{(2)}_\varrho(\widetilde{X})] \in \map(\R,\R) / \{ x \mapsto \max\{ \frac{x}{2} \cdot \ln(\varrho(g_1)) , \frac{x}{2} \cdot \ln(\varrho(g_2)) \} \mid g_1, g_2 \in G \}
\]
and have to show $[\rho^{(2)}_\varrho(\widetilde{X})] = 0$.
Since $[\rho^{(2)}_\varrho(\widetilde{S^1})] = 0$, we can apply the same arguments as in the proof of \cite[Theorem 3.105]{Luc02}.
\end{enumerate}
\end{proof}

\begin{remark} \label{rem_sum-formula}
We are mostly interested in $G$-CW-complexes given by the universal covering $\widetilde{X}$ of a finite CW-complex $X$ (here $G = \pi_1(X)$).
In this context the sum formula (Proposition \ref{properties-torsion} (\ref{properties-torsion-2})) carries over to the following statement:\\
Consider the pushout of finite CW-complexes such that $j_1$ is an inclusion of CW-complexes, $j_2$ is cellular and $X$ inherits its CW-complex structure from $X_0$, $X_1$ and $X_2$.
\begin{equation*}
\xymatrix{
X_0 \ar[r]^{j_1} \ar[d]_{j_2} & X_1 \ar[d]_{i_1} \\
X_2 \ar[r]^{i_2} & X }
\end{equation*}
Let $\varrho \colon \pi_1(X) \to \R^{>0}$ be a group homomorphism. We obtain induced group homomorphisms $\varrho_i \colon \pi_1(X_i) \to \pi_1(X) \stackrel{\varrho}{\longrightarrow} \R^{>0}$ ($i=0,1,2$).
Assume that $\widetilde{X_i}$ is det-$L^2$-acyclic w.r.t. $\varrho_i$ for $i=0,1,2$.
Then $\widetilde{X}$ is det-$L^2$-acyclic w.r.t. $\varrho$ and
\[
\rho^{(2)}_\varrho(\widetilde{X}) = \rho^{(2)}_{\varrho_1}(\widetilde{X_1}) + \rho^{(2)}_{\varrho_2}(\widetilde{X_2}) - \rho^{(2)}_{\varrho_0}(\widetilde{X_0}).
\]
\end{remark}

A way to construct out of
\[
\rho^{(2)}_\varrho(X) \in \map(\R,\R) / \{ x \mapsto \frac{x}{2} \cdot \ln(\varrho(g)) \mid g \in G \}
\]
a simplified invariant with values in $\R$ is described in the following definition.

\begin{definition}
Let $X$ be a finite free $G$-CW-complex and $\varrho \colon G \to \R^{>0}$ a group homomorphism. We define the simplified weighted $L^2$-torsion of $X$ by
\[
\tilde{\rho}^{(2)}_\varrho(X) := \frac{1}{2} \cdot \big( \rho^{(2)}_\varrho(X)(1) + \rho^{(2)}_\varrho(X)(-1) \big) \in \R
\]
\end{definition}

For $\varrho = 1$ we obtain $\tilde{\rho}^{(2)}_\varrho(X) = \rho^{(2)}(X)$.
Obviously, all statements of Theorem \ref{properties-torsion} carry over to this simplified weighted $L^2$-torsion.

\end{document}